\documentclass[11pt,reqno]{amsart}
\usepackage{bbm}
\usepackage{amssymb, amsthm, amsmath, amscd}
\usepackage[usenames,dvipsnames]{color}
\usepackage{xcolor}

\theoremstyle{definition}
\newtheorem{notation}{Notation}[section]
\newtheorem{theorem}[notation]{Theorem}
\newtheorem{proposition}[notation]{Proposition}
\newtheorem{corollary}[notation]{Corollary}
\newtheorem{lemma}[notation]{Lemma}
\newtheorem{remark}[notation]{Remark}
\newtheorem{definition}[notation]{Definition}

\newtheorem*{theorem*}{Theorem}
\newtheorem*{corollary*}{Corollary}
\newtheorem{example}[notation]{Example}

\newtheorem{theoremx}{Theorem}
\newtheorem{definitionx}[theoremx]{Definition}

\usepackage{amsfonts}
\usepackage{mathrsfs}
\usepackage{enumerate}
\usepackage{textcomp}
\usepackage[all]{xy}
\usepackage{enumitem}
\usepackage{mathtools}
\usepackage{soul}
\usepackage{hyperref}
\usepackage{cleveref}
\usepackage{tikz}
\usepackage{adjustbox}
\hypersetup{
	colorlinks=true,
	citecolor=Blue,
	linkcolor=Blue,
	pdfborder=0 0 0}

\numberwithin{equation}{section}
\begin{document}
	
	\title[Rokhlin dimension for actions of residually compact groups]{Rokhlin dimension for actions of residually compact groups}
	
	\author{Xin Cao}
	\address{School of Mathematical Sciences, Key Laboratory of Intelligent Computing and Applications(Ministry of Education), Tongji University, Shanghai 200092, China}
	\email{2111153@tongji.edu.cn}
	
	\author{Xiaochun Fang}
	\address{School of Mathematical Sciences,  Key Laboratory of Intelligent Computing and Applications(Ministry of Education), Tongji University, Shanghai 200092, China}
	\email{xfang@tongji.edu.cn}
	
	\author{Jianchao Wu}
	\address{Shanghai Center for Mathematical Sciences, Fudan University,\newline\indent 2005 Songhu Rd., Shanghai 200438, China}
	\email{jianchao\_wu@fudan.edu.cn}

	\date{Month, Day, Year}
	\date{\today}

    \thanks{The last author was partially supported by NSFC Key Program No.~12231005 and National Key R\&D Program of China 2022YFA100700. }

	\begin{abstract}
         We introduce the concept of Rokhlin dimension for actions of residually compact groups on C*-algebras, which extends and unifies previous notions for actions of compact groups, residually finite groups and the reals. We then demonstrate that finite nuclear dimension (respectively, absorption of a strongly self-absorbing C*-algebra) is preserved under the formation of crossed products by residually compact group actions with finite Rokhlin dimension (respectively, finite Rokhlin dimension with commuting towers). Furthermore, if second countable residually compact group contains a non-open cocompact closed subgroup, then crossed products arising from actions with finite Rokhlin dimension are stable. Finally, we study the relationship between the tube dimension of a topological dynamical system and the Rokhlin dimension of the induced C*-dynamical system.
	\end{abstract}
    \maketitle

    \tableofcontents
	
	\section{Introduction}\label{sec1}
	Given a C*-dynamical system $(G, A,\alpha )$, the Rokhlin property is an important regularity condition on $\alpha$, it can be used to transfer certain properties from $A$ to the crossed product C*-algebra $A\rtimes_\alpha G$. In \cite{Izu04a,Izu04b}, Izumi initiated his pioneering work on finite group actions with the Rokhlin property. Roughly speaking, for an action by a finite group, the Rokhlin property says that there exists a partition of unity consisting of approximately central projections that are mutually orthogonal, indexed by the elements of the group, and permuted by the group action. Such a partition of unity is a prototypical example of what one may call a Rokhlin tower in the C*-dynamical setting.
    There is a plethora of work on crossed products that are formed by actions with Rokhlin-type properties (see, e.g., \cite{BKR07,Gar17a,HW07,Kis02,MS12,MS14,Nak99,OP12,Phi11,San15}). However, group actions with the Rokhlin property are rare, and many C*-algebras do not admit finite group actions with the Rokhlin property, due to restrictions on their $K$-theory (see, e.g., \cite[Section 10.3]{GKPT18}).  
    
	\par 
        There are two major approaches to soften the definition of the Rokhlin property and expand its applicability.   
        The first approach allows a ``small remainder'' to complement a Rokhlin tower, and results in the notions of tracial Rokhlin property \cite{Phi11} and subsequently weak tracial Rokhlin property \cite{Arc08,Wang13,HO13}. 
        The main application of this approach is to show that tracially $\mathcal{Z}$-absorption is preserved by taking crossed products by actions with the weak tracial Rokhlin property.  
        However, this approach typically requires the simplicity of the C*-algebra and does not seem to work well with the notion of finite nuclear dimension.  
        
        This paper follows the second approach, which allows a partition of unity consisting of multiple Rokhlin towers, and results in the notion of the Rokhlin dimension \cite{HWZ15}. Roughly speaking, for a finite group action, a Rokhlin tower in this formulation consists of positive elements that are mutually orthogonal, indexed by the group elements, and permuted by the group action. It is easy to see that an action of a finite group has zero Rokhlin dimension if and only if it satisfies the Rokhlin property. 
        The notion of finite Rokhlin dimension generalizes Rokhlin property while still being strong enough to ensure compatibility with regularity properties such as finite nuclear dimension. 
    
	\par The Rokhlin dimension was initially defined for actions by finite groups and the group $\mathbb{Z}$ of integers \cite{HWZ15} and subsequently generalized to actions of more general groups: compact groups in \cite{Gar17a,Gar17b,GHS19}, residually finite groups in \cite{SWZ19}, and the group $\mathbb{R}$ of real numbers in \cite{HSWW17}. All of these definitions can be used to transfer good structure properties, such as finite nuclear dimension and $\mathcal{Z}$-absorption (the latter case requires an extra assumption of commuting towers), from a C*-algebra $A$ to its crossed product C*-algebra $A \rtimes_\alpha G$. 
    A natural next target would be residually compact groups, a concept that extends residual finiteness to the scope of locally compact groups. 
    To this end, Szab{\'o} \cite{Sza19} generalized the theory of Rokhlin dimension \emph{with commuting towers} to work for second countable locally compact groups with residually compact approximations that are decreasing, discrete and normal. 

        \par In this paper, we unify all the known constructions of Rokhlin dimension and generalize the theory to work for residually compact groups, defined as follows: a second countable locally compact group is called \emph{residually compact} if it admits a sequence, called a \emph{residually compact approximation}, that consists of closed cocompact subgroups such that the intersection of these subgroups is trivial. Our definition of a residually compact approximation is more general than those of Szab{\'o} \cite{Sza19} and Amini \cite{A24}, in that our sequence of cocompact closed subgroups need not be decreasing, nor discrete, nor normal. 
        Unlike in the residually finite setting, residual compactness for second countable locally compact groups does not automatically admit decreasing, discrete, or normal approximating subgroups, so it is desirable to forgo these requirements in the theory. 
        On the other hand, just like the case of residually finite groups, in order to apply the theory of Rokhlin dimension, we shall almost always assume a regularity condition on the residually compact approximations. The regular residually compact approximation in this paper includes, as a special case, the one in \cite{Sza19} which consists of decreasing discrete normal cocompact subgroups (see Proposition~\ref{prop: decr disc nor is regular}).
        Note that the various classes of groups mentioned in the last paragraph, namely compact groups, residually finite groups and the group $\mathbb{R}$ of real numbers, are all residually compact groups. Further natural examples, for which a Rokhlin dimension theory has not been fully established, include $\mathbb{R}^d$ and the continuous Heisenberg group, etc. 
        In fact, all these groups admit regular residually compact approximations. 
      
    \par We define the Rokhlin dimension for residually compact group actions as follows (here we assume unitality to ease the exposition):
    \begin{definitionx}[see Definition~\ref{def main} for the general case]
        Let $G$ be a residually compact group with a residually compact approximation $\sigma$. Let $A$ be a unital separable C*-algebra, and let $\alpha\colon G\to {\rm Aut}(A)$ be an action. Let $d\in \mathbb{N}$ be a natural number.
        \par For any $H\in \sigma$, we say $\alpha$ has \emph{Rokhlin dimension at most $d$ relative to $H$} and write $\mathrm{dim}_{\mathrm{Rok}}(\alpha,H)\leq d$, if there exist equivariant completely positive contractive order zero maps
        \[\varphi_l\colon C(G/H)\to A_\infty\cap A'\quad (l=0,\ldots,d) \]     
        with $\varphi_0(1)+\cdots+\varphi_d(1)=1$. The value ${\rm dim}_{\rm Rok}(\alpha,H) $ is defined to be the smallest $d\in\mathbb{N}$ such that ${\rm dim}_{\rm Rok}(\alpha,H)\leq d$, or $\infty$ if no such $d$ exists.
        \par We define the \emph{Rokhlin dimension} of $\alpha$ along $\sigma$ as
        \[ \mathrm{dim_{Rok}}(\alpha,\sigma)=\sup_{H\in \sigma}\mathrm{dim_{Rok}}(\alpha,H).\]
    \end{definitionx}

    This definition of Rokhlin dimension is compatible with those given earlier in various special cases (\cite{HWZ15,Gar17a,Gar17b,GHS19,SWZ19,HSWW17}). Imposing the extra requirement that the order zero maps $\varphi_0, \ldots, \varphi_d$ have commuting ranges, we arrive at the related notion of Rokhlin dimension with \emph{commuting towers}, which again generalizes the corresponding definitions in these previous papers as well as the one in \cite{Sza19}. 

    \par 
    To see how this definition matches the description given earlier that uses Rokhlin towers, we observe that each map $\varphi_l$ in the above definition can be lifted to an approximately equivariant, approximately central, approximately order zero, completely positive contractive map $\psi_l\colon C(G/H) \to A$ (see Proposition~\ref{prop5.1}). In particular, if $G$ is a finite group and $H$ is chosen to be trivial, then $\{\psi_l(\delta_g)\colon g\in G\}$ forms a Rokhlin tower in the previous descriptions, where each $\delta_g$ is the characteristic function of $\{g\}$. 
    
	\par With the above definition, we are able to unify and generalize the previous theories on Rokhlin dimension. In particular, we show a number of important C*-algebraic properties are preserved under taking crossed products with regard to actions with finite Rokhlin dimension. Let us start with finite nuclear dimension.  
	\begin{theoremx}[{see Theorem~\ref{thm:dimnuc}}]
		Let $A$ be a separable C*-algebra, let $G$ be a second countable residually compact group, and let $\alpha\colon G\to {\rm Aut}(A)$ be a continuous action. 
        Let $\sigma=(G_n)_{n\in \mathbb{N}}$ be a regular residually compact approximation of $G$. Then the following inequality holds:
		\[{\rm dim}_{\rm nuc}^{+1}(A\rtimes_\alpha G)\leq {\rm asdim}^{+1}(\Box_{\sigma}G)\cdot {\rm dim}_{\rm Rok}^{+1}(\alpha,\sigma)\cdot {\rm dim}_{\rm nuc}^{+1}(A). \]
	\end{theoremx}
    For notational convenience, each dimension in the formula is increased by 1, as indicated by the superscripts. 
    Here $\Box_{\sigma}G$ denotes the box space of $G$ associated to $\sigma$, that is, the coarse disjoint union of the sequence of quotient spaces $\{G/H\colon H\in \sigma\}$. We point out that the asymptotic dimension and the box space construction have individually played significant roles in coarse geometry. 
    This above formula directly generalizes the one for residually finite groups \cite[Theorem~B]{SWZ19}. To see how it recovers the nuclear dimension bounds for compact groups \cite[Theorem~3.4]{Gar17b} and the reals \cite[Theorem~4.5]{HSWW17}, one simply observes that ${\rm asdim}(\Box_{\sigma}G) = 0$ for any compact group $G$ while ${\rm asdim}(\Box_{\sigma} \mathbb{R}) = 1$. 
    
    Another regularity property of great significance in the modern theory of C*-algebras is $\mathcal{Z}$-stability, also known as $\mathcal{Z}$-absorption, that is, the ability to tensorially absorb the Jiang-Su algebra $\mathcal{Z}$. This fits into a range of interesting properties called $D$-absorption, where $D$ is a strongly self-absorbing C*-algebra, e.g., $\mathcal{Z}$, $M_{p^\infty}$, $\mathcal{O}_\infty$ and $\mathcal{O}_2$ (\cite{TW07}). For example, $\mathcal{O}_\infty$-absorption has played a pivotal role in the classification theory of purely infinite C*-algebras. 
   Again, we extend the previous theories by showing that, under the assumptions of finite Rokhlin dimension \emph{with commuting towers} and finite-dimensional box spaces, $D$-absorption is also preserved under taking crossed products (see \cite[Theorem~4.4]{Gar17a}, \cite[Theorem~F]{SWZ19}, \cite[Theorem~C]{HSWW17} and \cite[Corollary~B]{Sza19}). 
   \begin{theoremx}[{see Theorem~\ref{thm: D-ab}}]
   	Let $D$ be a strongly self-absorbing C*-algebra. Let $A$ be a separable, $D$-absorbing C*-algebra. Let $G$ be a second countable residually compact group with a regular approximation $\sigma=(G_n)_{n\in\mathbb{N}}$ and $\alpha\colon G\to {\rm Aut}(A)$ a continuous action. If ${\rm asdim}(\Box_{\sigma}G)<\infty$ and $\alpha$ has finite Rokhlin dimension with commuting towers, then $A\rtimes_\alpha G$ is $D$-absorbing.
   \end{theoremx}
   For non-unital C*-algebras, it is also meaningful to ask whether they are stable. For non-discrete locally compact groups, the resulting crossed products are always non-unital. We present a theorem concerning the stability of these crossed product C*-algebras. Even in the case of compact groups, this result is more general than \cite[Corollary~4.19]{GHS19} in that we do not require commuting Rokhlin towers.
   \begin{theoremx}[{see Theorem~\ref{thm: stable}}]
   		Let $G$ be a second countable, residually compact group with a non-open cocompact closed subgroup (in particular, $G$ is not discrete). Let $A$ be a separable C*-algebra and let $\alpha\colon  G\to {\rm Aut}(A)$ be a continuous group action. If ${\rm dim}_{\rm Rok}(\alpha)<\infty$, then $A\rtimes_\alpha G$ is stable.
   \end{theoremx}

    Our strategy of proof for this result deviates significantly from those for previous results of this kind, which is necessitated by the fact that all previous results rely on commutativity of some form (commuting towers in \cite[Corollary~4.19]{GHS19} and commutativity of the acting group in \cite[Theorem~7.6]{HSWW17}), a luxury we do not have. Our proof, like that of \cite[Theorem~7.6]{HSWW17}, also ends up verifying a local characterization of stability developed in \cite{HR98} for $\sigma$-unital C*-algebras (see Lemma~\ref{lem8.1}), which, very roughly speaking, says stability can be witnessed by finding a way to ``move'' each positive element to an approximately orthogonal position. However, deeper insights are needed to verify this local characterization in our generality. 
    
    To illustrate the ideas (in an informal way), we consider the special case where $A$ is unital and $G$ is compact. Note that $C(G)\rtimes G \cong \mathcal{K}(L^2(G))$, which is the ultimate source of stability in our proof. The task now is to transfer the ``stability witnesses'' from $C(G)\rtimes G$ to $A\rtimes_\alpha G$. 
    In the easiest case where ${\rm dim}_{\rm Rok}(\alpha)=0$, we have a $G$-equivariant $\ast$-homomorphism $\varphi \colon C(G) \to A_\infty \cap A'$, which then induces a $\ast$-homomorphism $C(G)\rtimes G \to (A\rtimes_\alpha G)_\infty$, the latter allowing us to directly map the ``stability witnesses'' of $C(G)\rtimes G$ into $A\rtimes_\alpha G$ (the subscript ``$\infty$'' is dropped through a perturbation argument). 
    However, this strategy does not generalize easily to higher Rokhlin dimensions, or for that matter, even the case where ${\rm dim}_{\rm Rok}(\alpha)=1$, since now we only have a (jointly unital) pair of $G$-equivariant order zero maps $\varphi_0, \varphi_1 \colon C(G) \to A_\infty \cap A'$, and the induced maps $C(G)\rtimes G  \to (A\rtimes_\alpha G)_\infty$  are also only order zero and thus not sufficient in themselves to preserve the structure of ``stability witnesses''. A key fact we exploit in our proof is that these induced order zero maps are bimodule maps with regard to the natural left and right $C^*(G)$-module structures on the crossed products. This becomes useful when combined with a more elaborate version of the aforementioned local characterization that we develop (see Lemma~\ref{lem8.2}), which uses certain orthogonal elements in a fixed subalgebra ($C^*(G)$ in this case) to witness stability of the ambient C*-algebra. Combining these ideas, one may transfer the ``stability witnesses'' of the pair $C(G)\rtimes G \supseteq C^*(G)$, via the above order zero maps, to $A\rtimes_\alpha G$, thus proving the latter is stable. 
   
   \par In \cite{HSWW17}, Hirshberg, Szab{\'o}, Winter and the last author of this paper also demonstrated a close relationship between the so-called tube dimension of a topological flow and the Rokhlin dimension of the induced C*-dynamical system. 
   In particular, it was shown that the Rokhlin dimension of a topological flow is bounded from above by the tube dimension, the latter being a notion formulated in purely topological dynamical terms. 
   We extend this result to the setting of amenable residually compact groups, building on top of the work of Enstad, Favre and Raum that generalizes the tube dimension to locally compact groups \cite[Definition~4.1]{EFR23}. 
   
   \begin{theoremx}[{see Theorem~\ref{thm: tubedim}}]
   	Let $G$ be an amenable residually compact group and let $X$ be a locally compact Hausdorff space. Let $G\curvearrowright X$ be an action and $\alpha\colon  G\to {\rm Aut}(C_0(X))$ be the associated action. Then
   	\[{\rm dim}_{\rm Rok}(\alpha)\leq {\rm dim}_{\rm tube}(G\curvearrowright X). \]
   \end{theoremx}
   
   Since it was shown in \cite[Theorem~4.2]{EFR23} that free actions of connected Lie groups of polynomial growth on finite-dimensional spaces admit finite tube dimension, we thus have a rich source of examples of actions with finite Rokhlin dimension coming from topological dynamical systems. 
   
	\par This paper is organized as follows: In Section~\ref{sec2}, we present some basic notions and preliminaries. In Section~\ref{sec3}, we discuss residually compact groups, extending the theory in \cite{A24}. In Section~\ref{sec4}, we define the box spaces of residually compact groups and prove a lemma that characterizes their asymptotic dimension. In Section~\ref{sec5}, we establish the definition of Rokhlin dimension for actions of residually compact groups. In Sections~\ref{sec6},~\ref{sec7} and~\ref{sec8}, we prove the theorems concerning finite nuclear dimension, $D$-absorption and stability of crossed product C*-algebras, respectively. In Section~\ref{sec9}, we study the relationship between tube dimension and the Rokhlin dimension.

	\section{Preliminaries}\label{sec2}
	\subsection{Some basic notations}

	\par Let $A$ be a C*-algebra. We denote by ${\rm Aut}(A)$ the automorphism group of $A$, and we write $\mathcal{M}(A)$ for its multiplier algebra. In addition, we denote by $A_+$ the set of all positive elements in $A$ and by $A_+^1$ the set of all positive contractive elements in $A$. Homomorphisms of C*-algebras are always assumed to be $\ast$-homomorphisms.
    
	\par Let $G$ be a locally compact group, an action of $G$ on $A$ is always assumed to be a (point-norm) continuous group homomorphism from $G$ to ${\rm Aut}(A)$. If $H$ is a subgroup of $G$, we write $G/H$ for its left coset space with the quotient topology. 
    
	\par We also adopt the following shorthand for the sake of brevity: if $\dim(-)$ is one of the notions of dimension (e.g., covering dimension $\dim(-)$, nuclear dimension ${\rm dim}_{\rm nuc}(-)$, asymptotic dimension ${\rm asdim}(-)$, Rokhlin dimension ${\rm dim}_{\rm Rok}(-)$, etc), we write $\dim^{+1}(-)$ in formulas to stand for $1+\dim(-)$, i.e., adding one to the value of the dimension.
    
	\par The abbreviations c.p.\ and c.p.c.\ stand for completely positive and completely positive contractive, respectively.
    
	\par Sometimes we write $a=_\varepsilon b$ to mean $\|a-b\|< \varepsilon$.

	\subsection{Central sequence algebras}(\cite{Kir06}) Let us recall some well-known definitions. 
    
	\par Let $A$ be a C*-algebra. We denote by $\ell^\infty(\mathbb{N},A)$ the set of all bounded sequences in $A$. This is a C*-algebra with the supremum norm and pointwise operations. Also define
	\[c_0(\mathbb{N},A)=\{(a_n)_{n\in\mathbb{N}}\in \ell^\infty(\mathbb{N},A)\colon \,\lim_{n\to\infty}\|a_n\|=0  \}. \]
	It is easy to see that $c_0(\mathbb{N},A)$ is an ideal in $\ell^\infty(\mathbb{N},A)$. We denote by $A_\infty$ the quotient $\ell^\infty(\mathbb{N},A)/c_0(\mathbb{N},A)$ and write $\eta_A\colon \ell^\infty(\mathbb{N},A)\to A_\infty$ for the quotient map. The C*-algebra $A$ can be embedded in $A_\infty$ as constant sequences, and we denote it by $\iota_A\colon A\to A_\infty$. 
    
	\par Let $D\subseteq A$ be a subalgebra. Define
	\[ A_\infty\cap D'=\{a\in A_\infty \colon\, ad=da \,\text{ for all }\, d\in D \}. \]
	It is the relative commutant of $D$ in $A_\infty$. Define 
	\[{\rm Ann}(D,A_\infty)=\{a\in A_\infty \colon \, ad=da=0 \,\text{ for all }\, d\in D \}. \]
	Then ${\rm Ann}(D,A_\infty)$ is a closed two-sided ideal in $ A_\infty\cap D'$. Thus, we set
	\[F(D,A_\infty)=( A_\infty\cap D')/{\rm Ann}(D,A_\infty), \]
	this is the central sequence algebra of $A$ relative to $D$. We write $\kappa_{D,A}\colon A_\infty\cap D'\to F(D,A_\infty)$ for the quotient map. If $D$ is $\sigma$-unital, then $F(D,A_\infty)$ is unital. If $D=A$, we write $F_\infty(A)=F(A,A_\infty)$. And, moreover, if $A$ is unital, $F_\infty(A)=A_\infty\cap A'$.
    
	\par Let $G$ be a locally compact group, $\alpha\colon G\to {\rm Aut}(A)$ be a continuous action of $G$ on $A$, and let $D$ be an $\alpha$-invariant subalgebra of $A$. Since every automorphism on $A$ can naturally induce an automorphism on $\ell^\infty(\mathbb{N},A)$ by being applied component-wise, $\alpha$ induces an action $\ell^\infty(\alpha)$ of $G$ on $\ell^\infty(\mathbb{N},A)$. Then there are also induced actions of $G$ on $A_\infty\cap D'$ and $F(D,A_\infty)$, which we denote by $\alpha_\infty$ and $\tilde{\alpha}_\infty$, respectively. However, $\ell^\infty(\alpha)$ may not be continuous in general. We consider the C*-subalgebra
    
	\[\ell^{\infty,(\alpha)}(\mathbb{N},A)=\{a\in\ell^\infty(\mathbb{N},A)\colon \, g\to \ell^\infty(\alpha_g)(a)\,\text{is continuous} \} .\]
    
	We set $A_\infty^{(\alpha)}=\eta_A(\ell^{\infty,(\alpha)}(\mathbb{N},A))$  and set $F^{(\alpha)}(D,A_\infty )=\kappa_{D,A}(A_\infty^{(\alpha)}\cap D')$. The subalgebras $A_\infty^{(\alpha)}$ and $F^{(\alpha)}(D,A_\infty )$ are invariant under $\alpha_\infty$ and $\tilde{\alpha}_\infty$, respectively. Then we can restrict $\alpha_\infty$ to $A_\infty^{(\alpha)}$ and $\tilde{\alpha}_\infty$ to $F^{(\alpha)}(D,A_\infty )$, which we also denote by $\alpha_\infty$ and $\tilde{\alpha}_\infty$, and the restrictions are continuous. Clearly, the image of the standard embedding of $A$ is in $A_\infty^{(\alpha)}$, so we can view $\iota_A$ as a map from $A$ to $A_\infty^{(\alpha)}$. 
    
	\begin{remark}
		We might also define 
		$$A_\infty^{(\alpha)}=\{a\in A_\infty\colon \,g\to \alpha_{\infty,g}(a) \,\text{is continuous} \}.$$
		In fact, by \cite[Theorem~2]{Bro00} these two definitions are consistent.
	\end{remark}
    
	\begin{remark}[\cite{Kir06}]\label{rem2.21}
		There is a canonical $\ast$-homomorphism 
		\[F(D,A_\infty)\otimes_{\rm max} D\to A_\infty\quad \text{via} \quad \kappa_{D,A}(a)\otimes d\mapsto a\cdot d .\]
		If $D$ is $\sigma$-unital, $1\otimes d$ is mapped to $d$ for all $d\in D$. Furthermore, the $\ast$-homomorphism is $\tilde{\alpha}_\infty\otimes\alpha|_D$-$\alpha_\infty$ equivariant, It can be restricted to an equivariant $\ast$-homomorphism
		\[F^{(\alpha)}(D,A_\infty)\otimes_{\rm max} D\to A^{(\alpha)}_\infty. \]
	\end{remark}
    
	We record the following useful lemma. 
    
	\begin{lemma}[{\cite[Lemma~2.3]{HSWW17}}] \label{lem2.21}
		Let $A$ be a C*-algebra, let $G$ be a locally compact group, and let $\alpha\colon  G\to {\rm Aut}(A)$ be a continuous action. Then there exists a natural $\ast$-homomorphism $\phi\colon  A_\infty^{(\alpha)} \rtimes_{\alpha_\infty}G \to (A\rtimes_\alpha G)_\infty$ that is compatible with the standard embedding in the sense that $\phi \circ(\iota_A \rtimes G)=\iota_{A\rtimes_\alpha G}$. 
	\end{lemma}

	\subsection{Order zero maps}
	We recall some basic definitions. More details can be found in \cite{WZ09}.
    
	\par Two elements $a$ and $b$ in a C*-algebra are said to be orthogonal, $a\perp b$, if $ab=ba=a^*b=b^*a=0$. If $a$ and $b$ are self-adjoint, then $a\perp b$ if and only if $ab=0$.
    
	\begin{definition}
		Let $A,B$ be C*-algebras and let $\varphi\colon A\to B$ be a c.p.c.\ map, we say that $\varphi$ has order zero if for any $a,b\in A_+$,
		\[a\perp b \Rightarrow \varphi(a)\perp\varphi(b). \]
	\end{definition}
    
	In the preceding definition, $a,b$ can be replaced by general elements in $A$, the result remains the same. Subsequently, applying \cite[Theorem~3.3]{WZ09}, we can easily deduce the following.
    
	\begin{proposition}
		Let $A,B$ be C*-algebras and let $\varphi\colon A\to B$ be a c.p.c.\ map, $\varphi$ has order zero if and only if
		\[\varphi(a)\varphi(bc)=\varphi(ab)\varphi(c) \]
		for every $a,b,c\in A$.
	\end{proposition}
    
	The following is another lemma that will be used throughout the paper.
    
	\begin{lemma}[{\cite[Lemma~2.4]{HSWW17}}]\label{lem2.31}
		Let $A$ and $B$ be C*-algebras, let $G$ be a locally compact group, and let $\alpha \colon G\to {\rm Aut}(A)$ and $\beta\colon  G\to{\rm Aut}(B)$ be continuous actions. Let $\varphi\colon  A\to B$ be an $\alpha$-$\beta$ equivariant c.p.c.\ order zero map. Then there is an induced c.p.c.\ order zero map $\varphi \rtimes G\colon A\rtimes_\alpha G\to B\rtimes_\beta G$.
		\par In fact, for every $k\in\mathbb{N}$, the full crossed product construction is functorial with respect to sums of $k$ c.p.c.\ order zero maps via the assignment $\sum_{j=1}^k\varphi^j\mapsto \sum_{j=1}^k \varphi^j\rtimes G$.
	\end{lemma}
	
	\subsection{Asymptotic dimension}
    The asymptotic dimension is an important invariant for metric spaces. Introduced by Gromov, it distinguishes itself from the classical notion of covering dimension in that it reflects the large-scale behaviors of metric spaces. We supply here only the basics, including one of the equivalent definitions. More details can be found in \cite[Chapter~9]{R03} and \cite[Chapter~2]{NY12}. See also \cite[Theorem~19]{BD08} and \cite[Proposition~1.2.7]{Sphd}.

    \begin{definition}
		The \emph{asymptotic dimension} of a metric space $(X,d)$, denoted by ${\rm asdim}(X)$, is defined to be the smallest non-negative integer $n$ such that for every $R>0$ there exists a uniformly bounded cover $\mathcal{V}=\mathcal{V}^0\cup\cdots\cup\mathcal{V}^n$ of $X$ with Lebesgue number at least $R$ such that each $\mathcal{V}^l$ consists of pairwise $R$-disjoint sets. 
	\end{definition}

    The idea that the asymptotic dimension reflects the large-scale behaviors of metric spaces can be made precise in saying that it is invariant under coarse equivalence (see  \cite[Chapter~2]{R03} and \cite[Chapter~1]{NY12}). 	\begin{definition}
		Let $(X,d_X)$ and $(Y,d_Y)$ be two metric spaces. A map $f\colon X\to Y $ is a \emph{coarse equivalence} if:
		\begin{enumerate}
			\item there is a $R>0$ such that 
			\[f(X)\cap B(y,R)\neq \emptyset,\, \text{for every} \,y\in Y, \]
			where $f(X)$ is the image and $B(y,R)$ is the open ball with $y$ as center and $R$ as radius.
			\item there are no-decreasing functions $\rho_-,\rho_+\colon [0,\infty)\to [0,\infty)$ with $\lim_{t\to\infty}\rho_-(t)=\infty=\lim_{t\to\infty}\rho_+(t)$ such that
			\[\rho_-(d_X(x,y))\leq d_Y(f(x),f(y))\leq \rho_+(d_X(x,y)) \]
			for all $x,y\in X$.
		\end{enumerate}
        
		\par Two metric spaces $(X,d_X)$ and $(Y,d_Y)$ are \emph{coarsely equivalent} if there exists a coarse equivalence map from $X$ to $Y$.
	\end{definition}
		
      \begin{lemma}\label{lem:asdim-coarse-eq}
          If two metric spaces $(X,d_X)$ and $(Y,d_Y)$ are coarsely equivalent, then ${\rm asdim} (X) = {\rm asdim} (Y)$. 
      \end{lemma}

	\section{Residually compact groups}\label{sec3}
	
	In this section, we introduce residually compact groups which are analogous to residually finite groups. To begin, let us recall some relevant definitions.
    
	\begin{definition}
		A discrete group $G$ is said to be \emph{residually finite} if for every $g\in G\backslash\{1_G\}$ there is a subgroup $H$ with finite index such that $g\notin H$.
		\par If moreover, $G$ is countable, then $G$ is residually finite if and only if there is a decreasing sequence $(G_n)_{n\in\mathbb{N}}$ of subgroups with finite index such that $\bigcap_{n\in \mathbb{N}}G_n=\{1_G\}.$
	\end{definition}
    
	For general topological groups, a natural generalization of the concept of finite-index subgroups is the following: 
    
	\begin{definition}
		Let $G$ be a topological group and $H$ be a subgroup of $G$. The subgroup $H$ is \emph{cocompact} if there exists a compact subset $K\subseteq G$ such that $G=KH$. 
	\end{definition}

	\begin{remark}
		A subgroup $H$ of $G$ is cocompact if and only if there exists a compact subset $K\subseteq G$ such that $G=HK$, since the left coset space and the right coset space are homeomorphic.
	\end{remark}

    For closed subgroups in a locally compact group, the following equivalent characterization is well-known. We provide a brief proof for completeness. 
    
	\begin{proposition}
		Let $G$ be a locally compact group and $H$ be a closed subgroup of $G$. Then $H$ is cocompact if and only if the quotient space $G/H$ is compact.
	\end{proposition}
    
	\begin{proof}
		If $H$ is closed cocompact, then the quotient space $G/H$ is compact, since the quotient map is continuous by definition. 
        
		\par For the converse, we denote the quotient map by $\pi$. Since $G$ is locally compact, for every $g\in G$ we can find an open neighborhood $U_g$ of $g$ with $\overline{U_g}$ compact. Then $\{\pi(U_g)\colon g\in G\}$ is an open cover of $G/H$. Thus, there is a finite open cover $\{\pi(U_i)\colon i=1,\ldots,n\}$ for some $n\in\mathbb{N}$. Let $K=\bigcup_{i=1}^n\overline{U_i}$. Then $K$ is compact and $G=KH$.
	\end{proof}

	Now we proceed to define residually compact groups.
     
	\begin{definition}\label{def1}
    A locally compact Hausdorff group $G$ is said to be \emph{residually compact} if for every $g\in G\backslash \{1_G\}$ there is a closed cocompact subgroup $H$ such that $g\notin H$.
	\end{definition}

    \begin{remark}
        It follows from a simple compactness argument that $G$ is residually compact if and only if for any compact subset $F\subseteq G\backslash \{1_G\}$, there exist $n\in \mathbb{N}$ and closed cocompact subgroups $H_1,\ldots,H_n \leq G$, such that $F\subseteq G\backslash \bigcap_{i=1}^n H_i$.
    \end{remark}
    
	With this definition, a second countable topological group can also be characterized in terms of sequences of closed cocompact subgroups. We first record a simple fact about locally compact groups.
    
	\begin{lemma}\label{lem3.1}
		Let $G$ be a locally compact group. Then $G$ is second countable if and only if it is $\sigma$-compact and first countable if and only if $G\backslash \{1_G\}$ is $\sigma$-compact.
	\end{lemma}
    
	\begin{proof}
        The first equivalence is standard (it can be proved directly or seen as a consequence of the Birkhoff-Kakutani theorem) and we omit its proof. We focus on the second equivalence. 
    
		Since $G$ is $\sigma$-compact, we have $G=\bigcup_{n\in \mathbb{}{N}}K_n$, where $K_n$ is compact. Since $G$ is first countable, $1_G$ has a countable neighborhood base $\{U_1,\ldots, U_n,\ldots\}$. Then, $K_n\backslash U_n$ is compact, for all $n\in \mathbb{N}$. Then, 
		\[\bigcup_{n\in \mathbb{N}}(K_n\backslash U_n)=\bigcup_{n\in \mathbb{N}}K_n\backslash \bigcap_{n\in \mathbb{N}}U_n=\bigcup_{n\in \mathbb{N}}K_n\backslash \{1_G\}=G\backslash \{1_G\} .\]
		Hence, $G\backslash \{1_G\}$ is $\sigma$-compact.
        
        \par If $G\backslash \{1_G\}$ is $\sigma$-compact,  then $G\backslash\{1_G\}=\bigcup_{n\in \mathbb{N}}K_n$, where $K_n$ is compact. By local compactness, we may assume there is an increasing sequence of open sets $U_n$ such that $K_n = \overline{U_n}$ for each $n$, and we may also fix a precompact open neighborhood $U$ of $1_G$. 
        A compactness argument then shows that $\{ U \setminus \overline{U_n} \}$ forms a countable local basis of $1_G$, while $G=\bigcup_{n\in \mathbb{N}}K_n \cup U$. 
		Hence, $G$ is first countable and $\sigma$-compact.
	\end{proof}
    
	\begin{proposition}
		Let $G$ be a second countable locally compact group. Then $G$ is residually compact if and only if there exists a sequence consisting of closed cocompact subgroups $G_n\leq G$ with $\bigcap_{n\in \mathbb{N}}G_n=\{1_G\}$.
	\end{proposition}
    
	\begin{proof}
		Since $G$ is second countable, it follows from Lemma~\ref{lem3.1} that $G\backslash\{1_G\}$ is $\sigma$-compact, that is, $G\backslash\{1_G\}=\bigcup_{n\in \mathbb{N}}K_n$, where $K_n$ is compact for all $n\in \mathbb{N}$. For every $K_n$, there exist $m_n\in \mathbb{N}$ and $H_{n_1},\ldots,H_{n_{m_n}}\leq G$ such that $K_n\subseteq G\backslash \bigcap_{i=1}^{m_n}H_{n_i}$. Therefore,
		\[ G\backslash\{1_G\}=\bigcup_{n\in \mathbb{N}}K_n\subseteq \bigcup_{n\in \mathbb{N}}(G\backslash\bigcap_{i=1}^{m_n}H_{n_i})=G\backslash (\bigcap_{n\in \mathbb{N}}\bigcap_{i=1}^{m_n}H_{n_i}).\]
		Then,
		\[\bigcap_{n\in \mathbb{N}}\bigcap_{i=1}^{m_n}H_{n_i}\subseteq \{1_G\}.  \]
		By rearranging, we obtain a countable sequence $\{H_n\}_{n\in \mathbb{N}}$ consisting of closed cocompact groups, and $\bigcap_{n\in\mathbb{N}}H_n=\{1_G\}$.
        
        \par For the converse, $\{G_n\}_{n\in\mathbb{N}}$ is a sequence consisting of closed cocompact subgroups with $\bigcap_{n\in\mathbb{N}}G_n=\{1_G\}$. So, $G\backslash\{1_G\}=\bigcup_{n\in\mathbb{N}}G_n^{\rm c}$. Then, for every compact subset $F\subseteq G\backslash\{1_G\}$, $F\subseteq \bigcup_{n\in\mathbb{N}}G_n^{\rm c}$. Then there exist $m\in\mathbb{N}$ and $n_1,\ldots,n_m\in \mathbb{N}$ such that $F\subseteq \bigcup_{i=1}^m G_{n_i}^{\rm c}=G\backslash\bigcap_{i=1}^m G_{n_i}$. Hence $G$ is residually compact.
	\end{proof}
    
     In general, the sequence can not be chosen as decreasing because the intersection of two cocompact groups may not be cocompact. For example, $\mathbb{Z}$ and $\sqrt{2}\mathbb{Z}$ are two closed cocompact subgroups of $\mathbb{R}$, yet their intersection is not cocompact. In \cite{A24}, Amini's definition needs the sequence to be decreasing, so our definition is slightly weaker than his. In the rest of the paper, our residually compact groups will always be assumed to be second countable, and the sequence is said to be a residually compact approximation.

     \begin{example}
         It is clear that the following groups are residually compact: the reals $\mathbb{R}$ (and more generally $\mathbb{R}^n$), compact groups and residually finite groups. 
     \end{example}
    
    Moreover, we can get more residually compact groups by semidirect products.
    
	\begin{proposition}
		Let $G,K$ be residually compact groups with residually compact approximations $(G_n)_{n\in\mathbb{N}},(K_n)_{n\in\mathbb{N}}$ respectively, and let $\alpha\colon G\to {\rm Aut}(K)$ be a group homomorphism with $(k,g)\to \alpha_g(k)$ being continuous where $k\in K,g\in G$. If for every $n$, $K_n$ is $\alpha|_{G_n}$-invariant,  then the semidirect product $K\rtimes_{\alpha}G$ is residually compact.
	\end{proposition}
    
	\begin{proof}
		Consider the sequence $\{K_n\rtimes_{\alpha}G_n\colon n\in\mathbb{N}\}$. For every $n\in\mathbb{N}$, $K_n\rtimes_{\alpha}G_n$ is a closed subgroup of $K\rtimes_{\alpha}G$ since $K_n$ is $\alpha|_{G_n}$-invariant. And it is obvious that $\bigcap_{n\in\mathbb{N}}(K_n\rtimes_{\alpha}G_n)=\{(1_K,1_G)\}$. So, it remains to prove that the subgroup $\{K_n\rtimes_{\alpha}G_n\}$ is cocompact. For every $n\in\mathbb{N}$, there exist compact subsets $F_n\subseteq G$ and $E_n\subseteq K$ such that $G=G_nF_n,K=K_nE_n$. Then $E_n\times F_n$ is a compact subset of $K\rtimes_{\alpha}G$. Moreover, for every $(k,g)\in K\rtimes_\alpha G$, there exist $g'\in G_n,f\in F_n$ such that $g=g'f$ and then there exist $k'\in K_n,e\in E_n$ such that $\alpha_{g'^{-1}}(k)=k'e$. Thus, $(k,g)=(\alpha_{g'}(k'),g')(e,f)$. That is to say, $K\rtimes_{\alpha}G=(K_n\rtimes_{\alpha}G_n)(E_n\times F_n)$.
	\end{proof}
    
	In particular, we have the following:
    
	\begin{corollary}
		Let $G, K$ be residually compact groups. Then the direct product $K\times G$ is residually compact.
	\end{corollary}
    
	\begin{corollary}
		Let $G$ be a residually compact group and $K$ be a compact group. Let $\alpha\colon G\to{\rm Aut}(K)$ be a group homomorphism with $(k,g)\to \alpha_g(k)$ being continuous. Then the semidirect product $K\rtimes_{\alpha}G$ is residually compact.
	\end{corollary}
    
	In addition to obtaining new residually compact groups through the semidirect product, we also have several propositions concerning subgroups.
    
	\begin{proposition}
		Let $G$ be a locally compact group and $H$ be a closed cocompact subgroup. If $H$ is residually compact, then $G$ is residually compact. 
	\end{proposition}
    
	\begin{proof}
		Since $H$ is a closed cocompact subgroup of $G$, there is a compact subset $F\subseteq G$ such that $G=FH$. We also have a residually compact approximation $(H_n)_{n\in\mathbb{N}}$ of $H$. For every $n\in \mathbb{N}$, there is a compact subset $E_n\subseteq H$ such that $H=E_nH_n$. Then $FE_n$ is a compact subset of $G$ and $G=(FE_n)H_n$. Furthermore, $H_n$ is closed in $G$ since $H$ is closed. Thus, $(H_n)_{n\in\mathbb{N}}$ is a residually compact approximation of $G$.
	\end{proof}
    
	\begin{proposition}[see {\cite[Lemma~7]{A24}}]\label{pro3.1}
		Let $G$ be a residually compact $\sigma$-compact group and let $H$ be an open normal cocompact subgroup of $G$. Then $H$ is residually compact.
	\end{proposition}
	
	\section{Box spaces and asymptotic dimension}\label{sec4}
    
	In this section, we generalize the construction of box spaces to the setting of residually compact groups, and study their asymptotic dimension. We start by recalling that if $G$ is a second countable locally compact group, there exists a right-invariant proper compatible metric $d$ on $G$ (\cite{Str74}). A metric is termed \emph{proper} if closed balls are compact, \emph{right-invariant} if $d(g_1h,g_2h)=d(g_1,g_2)$ for every $g_1,g_2,h\in G$, and \emph{compatible} if the topology induced by the metric coincides with the original topology. 
	\par Then, for a closed cocompact subgroup $H\leq G$,  $d$ induces a metric $d'$ on $G/H$ by
	\[d'(g_1H,g_2H)=\inf\{d(g_1h_1,g_2h_2)\colon h_1,h_2\in H\},\]
	and this metric generates the topology on $G/H$( \cite[Lemma~3.20]{DK18}). Extending the construction for residually finite groups, we now define the box space of a residually compact group.
    
	\begin{definition}\label{def:box_space}
		Let $G$ be a second countable, locally compact, residually compact group, and let $\sigma=(G_n)_{n\in\mathbb{N}}$ be a sequence of closed cocompact subgroups of $G$. Let us equip $G$ with a proper, right-invariant, compatible metric $d$. For each $n\in \mathbb{N}$, there is an induced quotient metric $d_n$ on $G/G_n$. Then the \emph{box space} $\Box_{\sigma}G$ associated with $\sigma$ is defined as the coarse disjoint union of the sequence of compact metric spaces $(G/G_n,d_n)$. More precisely, it is a coarse space whose underlying set is the disjoint union $\bigsqcup_{n\in\mathbf{N}}G/G_n$, and whose coarse structure is determined by a metric $d^B$ such that for all $n$, the metric $d^B$ restricts to $d_n$ on the subset $G/G_n$, and such that we have \[d^B(G/G_n, G/G_m)\to\infty\]
		as $n \neq m$ and $n+m\to\infty$.
	\end{definition}
    
	Up to coarse equivalence, the box space $\Box_{\sigma}G$ does not depend on the choice of the proper right-invariant compatible metric. The proof of the following two theorems are almost identical to those in the setting of residually finite groups (\cite[Lemma~1.2.9 and Proposition~1.2.10]{Sphd}) and thus are skipped.
    
	\begin{theorem}\label{thm:box_space-independent-group-metric}
		Let $G$ be a second countable, locally compact group. Then for any two proper, right-invariant, compatible metrics $d_1$ and $d_2$ on $G$. there exists a monotonically increasing function $\rho\colon [0,\infty)\to[0,\infty)$ with $\lim_{t\to\infty
		}\rho(t)=\infty$, and such that $d_2(g,h)\leq \rho(d_1(g,h))$ for all $g,h\in G$. Moreover, the metric space $(G,d_1)$ and $(G,d_2)$ are coarsely equivalent through the identity map.
	\end{theorem}
    
	\begin{theorem}\label{thm:box_space-independent-gap-metric}
		Let $G$ be a second countable, locally compact residually compact group, and let $\sigma=(G_n)_{n\in \mathbb{N}}$ be a residually compact approximation of $G$. Let $d^1,d^2$ be two proper, right-invariant, compatible metrics on $G$. For each $n\in\mathbb{N}$, let $d_n^1,d_n^2$ be the quotient metrics, respectively. Then the coarse disjoint unions of the sequences $(G/G_n,d_n^1)$ and $(G/G_n,d_n^2)$ are coarsely equivalent via the identity map.
	\end{theorem}

    \begin{corollary}
        Let $G$ be a second countable, locally compact, residually compact group, and let $\sigma=(G_n)_{n\in\mathbb{N}}$ be a sequence of closed cocompact subgroups of $G$. Then ${\rm asdim} (\Box_{\sigma}G)$ is well-defined, that is, it does not depend on the choice of the metrics in Definition~\ref{def:box_space}. 
    \end{corollary}

    \begin{proof}
        This follows directly from Theorem~\ref{thm:box_space-independent-group-metric}, Theorem~\ref{thm:box_space-independent-gap-metric},  and Lemma~\ref{lem:asdim-coarse-eq}. 
    \end{proof}
    
	\par Note that as in \cite{SWZ19}, we do not require our approximation sequences to consist only of normal subgroups, a departure from the classical literature on the subject of box spaces. Herein lies a caveat: for a general non-normal approximation sequence $\sigma$, the coarse structure of the box space $\Box_\sigma G$ may behave badly with regard to that of $G$ itself. To rule out certain pathological examples, a regularity condition was introduced in \cite[Notation~3.6 and Lemma~3.7]{SWZ19}. We shall need a suitable generalization of this condition. 
    
	\begin{definition}
		A sequence $\sigma$ consisting of closed cocompact subgroups of $G$ is called \emph{regular} if, for any compact subset $F\subseteq G$, there is $H\in \sigma$ such that the quotient map $G\to G/H$ is injective on $F\cdot k$ for any $k\in G$. 
	\end{definition}
    
	\begin{remark}
		An approximation sequence $\sigma$ is regular if and only if for any compact subset $F\subseteq G$, there is $H\in \sigma$ such that $F^{-1}F\cap \bigcup_{k\in G}kHk^{-1}=\{1_G\}$. To prove this, one can directly adapt the argument in \cite[Lemma~3.7]{SWZ19}. 
	\end{remark}
    
    In \cite{Sza19}, Szab{\'o} studied second countable locally compact groups equipped with a residually compact approximation, which consists of decreasing discrete normal cocompact subgroups. We point out that the residually compact approximation in \cite[Definition~2.1]{Sza19} is a regular residually compact approximation in this paper. Moreover, our regular residually compact approximation is more general than his.
	\begin{proposition}\label{prop: decr disc nor is regular}
		Let $G$ be a locally compact group and $\sigma$ be a decreasing residually compact approximation consisting of normal discrete cocompact subgroups. Then $\sigma$ is regular.
	\end{proposition}
    \begin{proof}
        For every compact set $F\subseteq G$, we see that $F^{-1}F$ is compact. For every $x\in F^{-1}F\backslash\{1_G\}$, there exists $H_x\in \sigma$ such that $x\notin H_x$. Since $H_x$ is closed, there is an open neighborhood $O_x$ of $x$ such that $O_x\cap H_x=\emptyset$. Since $H_x$ is discrete, there is an open neighborhood $U_x(1_G)$ of $1_G$ such that $U_x(1_G)\cap H_x=\{1_G\} $. Then $\bigcup_{x\in F^{-1}F}(O_x\cup U_x(1_G))$ is an open cover of $F^{-1}F$. Hence we obtain a finite cover and we can find $H\in\sigma$ such that $F^{-1}F\cap  H=\{1_G\} $ since $\sigma$ is decreasing.
    \end{proof}
    
	The following lemma is important for understanding why we define regular residually compact approximations in this way.
    
	\begin{lemma}[{\cite[Lemma~3.11]{SWZ19}}]\label{lem4.1}
		Let $X$ be a metric space with an isometric right action of a group $H$, let $(Y,d_Y)$ be the quotient space by the group action with the induced pseudo-metric, and let $\pi\colon  X\to Y$ be the corresponding projection. Let $x\in X$ and $R>0$ be such that $\pi$ is injective when restricted to $B_{3R}(x)$. Then $\pi|_{B_R(x)}$ is an isometry from $B_R(x)$ onto $B_R(\pi(x))$.
	\end{lemma}
    
	In Section~\ref{sec3}, we discussed how to obtain new residually compact groups. Now, if the residually compact approximation needs to be regular, do these results still hold? The answer is yes.
    
	\begin{proposition}
		Let $G$ be a residually compact group with a regular decreasing residually compact approximation $(G_n)_{n\in\mathbb{N}}$ and $K$ be a residually compact group with a regular decreasing residually compact approximation $(K_n)_{n\in\mathbb{N}}$. Let $\alpha\colon  G\to {\rm Aut}(K)$ be a group homomorphism with $(k,g)\to \alpha_g(k)$ being continuous and for every $n\in\mathbb{N}$, the subgroup $K_n$ is $\alpha|_{G_n}$ invariant. Then $(K_n\rtimes_{\alpha} G_n)_{n\in\mathbb{N}}$ is a regular residually compact residually approximation of $K\rtimes_\alpha G$.
	\end{proposition}
    
	\begin{proof}
		We denote by $\pi_K$ the canonical map from $K \rtimes_\alpha G$ to $K$ and $\pi_G$ the canonical map from $K \rtimes_\alpha G$ to $G$.  Given any compact set $F\subseteq K\rtimes_\alpha G $, $\pi_K(F)$ is compact in $K$. There exists $n_1\in\mathbb{N}$ such that $\pi_K(F)^{-1}\pi_K(F)\cap kK_{n_1}k^{-1}=\{1_K\}$ for every $k\in K$. Similarly, there exists $n_2\in\mathbb{N}$ such that $\pi_G(F)^{-1}\pi_G(F)\cap gG_{n_2}g^{-1}=\{1_G\}$ for every $g\in G$. Let $n=\max\{n_1,n_2\}$, we claim that
        
        \[ F^{-1}F\cap (k,g)(K_n\rtimes_\alpha G_n)(k,g)^{-1}=\{(1_K,1_G)\}\]
        
		for every $(k,g)\in K\rtimes_\alpha G$.
        
		\par The elements of $ F^{-1}F$ have the form that 
		\[ (\alpha_{g_1^{-1}}(k_1^{-1}k_2),g_1^{-1}g_2)\]
		for some $k_1,k_2\in\pi_K(F)$ and $g_1,g_2\in\pi_G(F)$.
		The elements in $(k,g)(K_n\rtimes_\alpha G_n)(k,g)^{-1} $ have the form
        \[(k\alpha_g(x)\alpha_{gyg^{-1}}(k^{-1}),gyg^{-1}) \]
		for every $k\in K,g\in G$ and some $x\in K_n,y\in G_n$. 
        
        \par For any $(k,g)\in K\rtimes_\alpha G$, if there are some $k_1,k_2,g_1,g_2,x,y$ as above such that \[(\alpha_{g_1^{-1}}(k_1^{-1}k_2),g_1^{-1}g_2)=(k\alpha_g(x)\alpha_{gyg^{-1}}(k^{-1}),gyg^{-1}). \]
		Then $g_1^{-1}g_2=gyg^{-1}$ for some $g_1^{-1}g_2\in\pi_G(F)^{-1}\pi_G(F)$ and $gyg^{-1}\in gG_ng^{-1}$. Thus, $g_1=g_2$ and $y=1_G$. Moreover $k\alpha_g(x)\alpha_{gyg^{-1}}(k^{-1})=k\alpha_g(x)k^{-1}$. Hence \[\alpha_{g_1^{-1}}(k_1^{-1}k_2)=k\alpha_g(x)k^{-1}.\]
		Then \[k_1^{-1}k_2=\alpha_{g_1}(k)\alpha_{g_1g}(x)\alpha_{g_1}(k)^{-1} \]
		for some $k_1^{-1}k_2\in \pi_K(F)^{-1}\pi_K(F) $ and some $\alpha_{g_1}(k)\alpha_{g_1g}(x)\alpha_{g_1}(k)^{-1}\in kK_nk^{-1}$. Thus, $k_1=k_2$ and the claim is proved.
	\end{proof}

	The following is the main lemma of this section. The proof is similar to the discrete case (\cite[Lemma 3.13]{SWZ19}).
    
	\begin{lemma}\label{lem main}
		Let $G$ be a second countable residually compact group with a proper, right-invariant compatible metric. Let $\sigma=(G_n)_{n\in \mathbb{N}}$ be a regular residually compact approximation. Let $s\in\mathbb{N}$.
        
        If ${\rm asdim}(\Box_{\sigma}G)\leq s$. Then, for every $\varepsilon>0$ and compact subset $M\subseteq G$, there exist $n$ and compact supported functions $\mu_l\colon G\to [0,1]$ for all $l=0,\ldots,s$ satisfying the following properties:
			\begin{enumerate}[label=(\alph*)]
				\item \label{lem main: disjoint} for every $l=0,\ldots,s$, one has 
				\[{\rm supp}(\mu_l)\cap {\rm supp}(\mu_l)h=\emptyset \quad\text{ for all }\, h\in G_n\backslash\{1_{G_n}\};\]
                
				\item\label{lem main: 1} for every $g\in G$, one has
				\[\sum_{l=1}^s\sum_{h\in G_n}\mu_l(gh)=1; \]
                
				\item\label{lem main: almost invar} for every $l=0,\ldots,s$ and $g\in M$, one has\[\|\mu_l-\mu_l(g\cdot-)\|_{\infty}\leq \varepsilon. \]
			\end{enumerate}

            Moreover, if $\sigma$ is decreasing, then ${\rm asdim}(\Box_\sigma G)\leq s$ if and only if the above conditions are satisfied.
	\end{lemma}
    
	\begin{proof}
		For each $n\in \mathbb{N}$, we denote by $\pi_n\colon G\to G_n$ the quotient map. Let $d$ be a proper, right-invariant compatible metric on $G$.

        let $\varepsilon>0$ and $M\subseteq G$ be a compact subset. Then there exists $R'>0$ such that $M$ is contained in $B_{R'}(1_G)$. Choose $R>\frac{2(2s+3)R'}{\varepsilon}$. Then, there is an open cover $\mathcal{V}=\mathcal{V}^0\cup\cdots\cup\mathcal{V}^s $ of $\Box_{\sigma}G$ with Lebesgue number at least $R$, such that the diameters of the members of $\mathcal{V}$ are uniformly bounded by some $R'>R$ and each $\mathcal{V}^l$ has mutually $R$-disjoint members. Since $\sigma$ is regular, for the compact subset $B_{3R'}(1_G)$ of $G$, there exists $n\in \mathbb{N}$ such that $\pi_n\colon G\to G/G_n$ is injective on $B_{3R'}(g)$ for every $g\in G$. It then follows from Lemma~\ref{lem4.1} that $\pi_n$ restricts to an isometry between $B_{R'}(g)$ and $B_{R'}(\pi_n(g))$ for any $g\in G$. We also have that 
		\[B_{3R'}(1_G)\cap \bigcup_{k\in G}kG_nk^{-1}\subseteq B_{3R'}(1_G)^{-1}B_{3R'}(1_G)\cap \bigcup_{k\in G}kG_nk^{-1}=\{1_G\} \]
        
		\par Let $\mathcal{V}'=\mathcal{V}'^0\cup\cdots\cup\mathcal{V}'^s $ denote the induced finite open cover on the compact subspace $G/G_n$. By a uniformly bounded condition, $V\in \mathcal{V}'$ is contained in $B_{R'}(\pi_n(g))$ for some $g\in G$, and thus we can find $U_V\subset B_{R'}(g)$ which is isometrically mapped to $V$ under $\pi_n$. Therefore, $\pi_n^{-1}(V)=\bigsqcup_{h\in G_n}U_Vh$. Define
		\[U^l \coloneqq\bigcup_{V\in\mathcal{V}'^l}U_V, \]
		\[\mathcal{U}^l\coloneqq \{U^lh\colon h\in G_n\} \]
		for all $l=0,\ldots,s$. We claim that $\mathcal{U}=\mathcal{U}^0\cup\cdots\cup\mathcal{U}^s$ is a uniformly bounded cover of $G$ with Lebesgue number at least $R$ and each $\mathcal{U}^l$ has mutually $R$-disjoint members. 
        
		\par Since $\mathcal{V}'$ is a finite cover, $U^l$ is bounded and $\mathcal{U}$ is uniformly bounded. Now consider the distance of elements in each $\mathcal{U}^l$. For all $h_1,h_2\in G_n$ and $g_1,g_2\in U_V\subseteq U^l$, we have
		\[d(g_1h_1,g_2h_2)=d(g_1h_1h_2^{-1}g_1^{-1},g_2g_1^{-1})>d(g_1h_1h_2^{-1}g_1^{-1},1_G)-2R'>R'\geq R \]
		because $g_1h_1h_2^{-1}g_1^{-1}\in g_1G_ng_1^{-1}$. For all $h_1,h_2\in G_n$ and $g_1\in U_{V_1}\subseteq U^l,g_2\in U_{V_2}\subseteq U^l$, we have
		\[d(g_1h_1,g_2h_2)\geq\inf\{d(g_1h_1',g_2h_2')|h_1',h_2'\in G_n \}=d_n(\pi_n(g_1),\pi_n(g_2))>R \]
		because $\pi_n(g_1)\in V_1\in \mathcal{V}'^l,\pi_n(g_2)\in V_2\in \mathcal{V}'^l$. So we obtain $d(U^lh_1,U^lh_2)>R$. 
        
		\par For any $x\in G$, $\pi_n(x)\in V$ for some $l\in \{0,\ldots,s\}$ and $V\in\mathcal{V}^l$. Then
		\[x\in \pi_n^{-1}(\pi_n(x))\subseteq \pi_n^{-1}(V)=\bigsqcup_{h\in G_n}U_Vh. \]
		Therefore, $x\in U_Vh_0\subseteq U^lh_0$ for some $h_0\in G_n$. Thus, $\mathcal{U}$ is a cover of $G$.
        
		\par Let $X\subseteq G$ be a subset with diameter less than $R$. The image $\pi_n(X)$ is a subset of $G/G_n$ with diameter less than $R$, so there exists some $l\in \{0,\ldots,s\}$ $V\in\mathcal{V}'^l$ such that $\pi_n(X)\subseteq V$, whence\[X\subseteq \pi_n^{-1}(\pi_n(X))\subseteq \pi_n^{-1}(V)=\bigsqcup_{h\in G_n}U_Vh. \]
		We showed that the $d(U_Vh_1,U_Vh_2)>R$ for all $h_1\neq h_2\in G_n$ and $V\in\mathcal{V}'$. Hence $X\subseteq U_Vh_0\subseteq U^lh_0$ for some $h_0\in G_n$. So, the Lebesgue number of the cover $\mathcal{U}$ is less than $R$.  The claim is shown.
		
		 Now define
        
		\[\mu_l\colon G\to [0,1], g\mapsto \frac{d(g,G\backslash U^l)}{\sum_{V\in \mathcal{U}}d(g,G\backslash V)} \]
        
		for all $l=0,\ldots,s$. It is easy to see that for every $l=0,\ldots,s$, ${\rm supp}(\mu_l)=\overline{U^l}$ ,\[{\rm supp}(\mu_l)\cap {\rm supp}(\mu_l)h=\emptyset \quad\text{ for all }\quad h\in G_n\backslash\{1_{G_n}\}.\]
		Furthermore, the property $(b)$ is also obvious. For property $(c)$, for every $l=0,\ldots,s$, $g\in M$ and $h\in G$,
		\begin{align*}
			|\mu_l(h)-\mu_l(gh)| &= |\frac{d(h,G\backslash U^l)}{\sum_{V\in \mathcal{U}}d(h,G\backslash V)}-\frac{d(gh,G\backslash U^l)}{\sum_{V\in \mathcal{U}}d(gh,G\backslash V)}|\\
			&\leq \frac{|d(h,G\backslash U^l)-d(gh,G\backslash U^l)|}{\sum_{V\in \mathcal{U}}d(h,G\backslash V)}+\frac{d(gh,G\backslash U^l)\sum_{V\in \mathcal{U}}|d(gh,G\backslash V)-d(h,G\backslash V)|}{\sum_{V\in \mathcal{U}}d(h,G\backslash V)\sum_{V\in \mathcal{U}}d(gh,G\backslash V)}\\
			&\leq \frac{d(h,gh)}{\sum_{V\in \mathcal{U}}d(h,G\backslash V)}+\frac{\sum_{V\in \mathcal{U}}|d(gh,G\backslash V)-d(h,G\backslash V)|}{\sum_{V\in \mathcal{U}}d(h,G\backslash V)}\\
			&\leq \frac{2}{R}d(h,gh)+\frac{2}{R}\sum_{V\in \mathcal{U}}|d(gh,G\backslash V)-d(h,G\backslash V)|\\
			&\leq \frac{2}{R}d(h,gh)+\frac{2}{R}(s(s+1))d(h,gh)\\
			&\leq \frac{2(2s+3)}{R}d(h,gh)\\
			&\leq \frac{2(2s+3)}{R}d(1_G,g)\leq \varepsilon.
		\end{align*}

        Now, if $\sigma$ is decreasing.  Given $R>0$, pick $\varepsilon>0$ with $\varepsilon<1/(s+1)$. For $\varepsilon $ and $M=B_R(1_G)$, choose $n$ and compactly supported continuous functions $\mu_l\colon G\to[0,1]$ for all $l=0,\ldots,s$ that satisfy the requirements of (3). By choosing $n$ large enough, we can also assume that the distance between any two of the subsets $\bigcup_{k=1}^{n-1}G/G_k$ and $G/G_m$ for all $m\geq n$ is at least $2R$ in $\Box_\sigma G$. Define $U^l\coloneqq {\rm supp}(\mu_l)$. Then $U^l\cap U^lh=\emptyset$ for all $l=0,\ldots,s$ and $h\in G_n\backslash\{1_G\}$, and $\{U^lh\colon h\in G_n,l=0,\ldots,s\}$ covers $G$.
        
		\par Since for any $g\in G$, $g$ is in the support of at most $s+1$ members of the partition of unity $\{\mu_l(- \cdot h)\colon h\in G_n,l=0,\ldots,s \}$, it follows that there exist $h\in G_n$ and $l\in\{0,\ldots,s\}$ such that $\mu_l(gh^{-1})\geq 1/(s+1)$. Thus, for any $g'\in B_R(g)$,
		\[\mu_l(g'h^{-1})\geq \mu_l(gh^{-1})-\varepsilon>0 .\]
		It is to say that for any $g\in G$, there are $h\in G_n$ and $l\in\{0,\ldots,s\}$ such that $B_R(g)\subseteq U^lh.$
        
		\par Note that, since $\pi_n$ is injective when restricted to $U^lh$ for all $h\in G_n$ and $l=0,\ldots,s$, we have $\pi_n$ is injective when restricted to $B_R(g)$ for any $g\in G$. For any compact subset $F\subseteq G$, there exists $R'>0$ such that $F\subseteq B_{R'}(1_G)$. Hence, for $\varepsilon>0$ and $M=B_{R'}(1_G)$, we have that $\pi_n$ is injective when restricted to $Fk$ for every $k\in G$. Thus, $\sigma$ is regular.
        
		\par Now for each $m\geq n$, define the collection $\mathcal{V}_m^l\coloneqq\{\pi_m(U^lh)\colon h\in G_n \}$, which consists of disjoint subsets of $G/G_m$. Define a cover $\mathcal{V}=\mathcal{V}^0\cup\cdots\cup\mathcal{V}^s$ of $\Box_\sigma G$ by 
        
		\[\mathcal{V}^0\coloneqq\{\bigcup_{k=1}^{n-1}G/G_k\}\sqcup\bigcup_{m=n}^\infty\mathcal{V}_m^0 ;\]
		\[\mathcal{V}^l\coloneqq\bigcup_{m=n}^\infty\mathcal{V}^l_m, l=1,\ldots,s. \]
        
		The diameters of these members are bounded by 
		\[{\rm max}\{{\rm diam}(\bigcup_{m=1}^{n-1}G/G_m),{\rm diam}(U^0),\ldots,{\rm diam}(U^s) \} \] 
		and each $\mathcal{V}^l$ consists of disjoint sets. Finally, let us show that its Lebesgue number is at least $R$. 
        
		\par Given a nonempty subset $X\subseteq \Box_\sigma G$ with diameter at most $R$, if we fix any $x\in X$, then $X\subseteq B_R(x)$. By our choice of $n$, we know that $B_R(x)$ falls entirely in one of the subsets $\bigcup_{k=1}^{n-1}G/G_k$ or $G/G_m$ for some $m\geq n$. In the first case, $B_R(x)\subseteq \bigcup_{k=1}^{n-1}G/G_k\in\mathcal{V}^0$. In the case where $B_R(x)\subseteq G/G_m$ for some $m\geq n$, there exists $g\in G$ such that $\pi_m(g)=x$ and $\pi_m(B_R(g))=B_R(x)$. For $B_R(g)$, we showed that there are $h\in G_n$ and $l\in\{0,\ldots,s\}$ such that $B_R(g)\subseteq U^lh$. Hence, 
		\[X\subseteq B_R(x)=\pi_m(B_R(g))\subseteq \pi_m(U^lh)\in \mathcal{V}_m^l. \]
        
		Therefore, $\mathcal{V}=\mathcal{V}^0\cup\cdots\cup\mathcal{V}^s$ is a uniformly bounded cover of $\Box_\sigma G$ with Lebesgue number at least $R$ and each $\mathcal{V}^l$ made up of disjoint members, which shows that ${\rm asdim}(\Box_\sigma G)\leq s$.
	\end{proof}
    
	\begin{corollary}
		Let $G$ be a second countable, locally compact, residually compact group, and $\sigma=(G_n)_{n\in\mathbb{N}}$ be a regular approximation. Let $G$ be equipped with some proper, right-invariant, compatible metric. Then ${\rm asdim}(G)\leq {\rm asdim}(\Box_{\sigma}G)$.
	\end{corollary}
    
	\begin{proof}
		This follows directly from the proof of Lemma~\ref{lem main}.
	\end{proof}
    
	\begin{corollary}
		Let $G$ be a second countable, locally compact, residually compact group with a proper, right-invariant, compatible metric, and let $\sigma=(G_n)_{n\in\mathbb{N}}$ be a regular approximation. Let $H\leq G$ be an open normal subgroup. Then
		\begin{enumerate}
			\item[(1)] $\kappa=(H\cap G_n)_{n\in\mathbb{N}}$ defines a regular approximation of $H$.
			\item[(2)] we have ${\rm asdim}(\Box_{\kappa}H)\leq {\rm asdim}(\Box_{\sigma}G)$.
		\end{enumerate}
	\end{corollary}
    
	\begin{proof}
		\par(1). By Proposition~\ref{pro3.1}, we know that $\kappa$ is a residually compact approximation of $H$. So we only need to prove that $\kappa$ is regular. For any compact $F\subseteq H$, $F$ is a compact set of $G$. Then there is $G_n\in\sigma$ such that $F^{-1}F\cap\bigcup_{k\in G}kG_nk^{-1}=\{1_{G_n}\}$. Thus 
		\[F^{-1}F\cap\bigcup_{k\in H}k(H\cap G_n)k^{-1}\subseteq F^{-1}F\cap\bigcup_{k\in G}kG_nk^{-1}=\{1_G\}. \]
		It is to say that $\kappa$ is regular.
		\par(2). This follows directly from Lemma~\ref{lem main} as well.
	\end{proof}

    We point out in passing that a prerequisite for a residually compact group to have a box space with finite asymptotic dimension is amenability. 
    
	\begin{definition}[{see, for example, \cite[Proposition~A.16]{Will07}}]\label{def:amenable}
		A locally compact group $G$ is \emph{amenable} if for every $\varepsilon>0$ and every compact subset $M\subseteq G$, there is a $f$ in the unit sphere of $ L^1(G)$ with $f\geq 0$ almost everywhere such that $\|f(x\cdot -)-f\|_1 \leq \varepsilon$ for every $x\in M$.
	\end{definition}
    
	\begin{corollary}\label{cor4.15}
		Let $G$ be a second countable residually compact group with a regular residually compact approximation $\sigma$ and ${\rm asdim}({\Box}_\sigma G)<\infty$. Then $G$ is amenable.
	\end{corollary}
    
	\begin{proof}
        Write $s = {\rm asdim}(\Box_\sigma(G))$ and let $\nu$ be a Haar measure on $G$. 
        Let $\varepsilon > 0$ and a compact subset $M$ in $G$ be given. By enlarging $M$ we may assume without loss of generality that $\nu(M) > 0$. Set $\delta = \frac{\min\{1, \varepsilon \nu(M)\}}{2(s+1)} > 0$ and apply Lemma~\ref{lem main} to obtain compactly supported functions $\mu_0,\ldots,\mu_s\colon G\to [0,1]$ that satisfy conditions~\ref{lem main: disjoint}, \ref{lem main: 1}, and \ref{lem main: almost invar} in Lemma~\ref{lem main} with $\varepsilon$ replaced by $\delta$. Apply condition~\ref{lem main: 1} to $g=1_G$ to see that there exists $l\in \{0,\ldots,s\}$ such that $\sum_{h\in G_n}\mu_l(h)\geq \frac{1}{s+1}$. Then by condition~\ref{lem main: disjoint}, there exists $h_0\in G_n$ such that $\mu_l(h_0)\geq \frac{1}{s+1}$. Through a right-translation of $\mu_l$, we may assume without loss of generality that $\mu_l(1_G)\geq \frac{1}{s+1}$. 
        It follows from condition~\ref{lem main: almost invar} that $\mu_l|_{M} \geq \frac{1}{2(s+1)}$ and thus $\|\mu_l\|_1 \geq \frac{\nu(M)}{2(s+1)}$. 
        Hence $f := \frac{\mu_l}{\|\mu_l\|_1} $ satisfies the conditions in Definition~\ref{def:amenable}. 
\end{proof}
    
  In the final part of this section, we will show that if a residually compact group with polynomial growth has a residually compact approximation consisting of discrete normal cocompact subgroups, then the associated box space has finite asymptotic dimension. Here, a residually compact group $G$ is said to be of \emph{polynomial growth} if there exist $C>0$ and $d\in\mathbb{N}$ such that $m(B_r(1_G))\leq Cr^d$ for every $r\in\mathbb{R}$, where $m$ is the right Haar measure. The discrete case can be found in \cite{DT18}.
  
  \begin{theorem}
  	Let $G$ be a residually compact group with polynomial growth and $\sigma=(G_n)_{n\in\mathbb{N}}$ be a decreasing residually compact approximation consisting of normal discrete cocompact subgroups. Then ${\rm asdim}(\Box_\sigma G)<\infty$.
  \end{theorem}
  
\begin{proof}
	Let $m$ be a right Haar measure of $G$, let $\pi_n$ be the quotient map from $G$ to $G/G_n$ and $\nu_n$ be a right Haar measure on $G/G_n$ for every $n\in\mathbb{N}$.
	\par Since $G$ is of polynomial growth, there exist $C>0$ and $d\in\mathbb{N}$ such that $m(B_r(1_G))\leq Cr^d$ for every $r\in\mathbb{R}$. Let $K=4^d+1$, for every $R>0$ there exists $t\in\mathbb{N}$ such that $(\frac{K}{4^d})^t\geq \frac{CR^d}{m(B_R(1))}$. Take $s_0=4^{t+1}R$, we claim that there exists $R'$ such that $R \leq R'\leq \frac{s_0}{4}$ and $m(B_{4R'}(1_G))\leq Km(B_{R'}(1_G))$. Indeed, if it is not true, then 
	\[C4^{td}R^d\leq K^tm(B_R(1_G))<m(B_{4^tR}(1_G))\leq C4^{td}R^d. \]
	It is impossible, so the claim is proved.
    
	\par Subsequently, by Lemma~\ref{lem4.1}, Proposition~\ref{prop: decr disc nor is regular} and the definition of Box space, there exists a large enough $N\in\mathbb{N}$ such that if $n\geq N$, then $\pi_n|_{B_{s_0}(x)}$ is an isometry from $B_{s_0}(g)$ to $B_{s_0}(\pi_n(g))$ for any $g\in G$ and the distance between any two of the subsets $\bigcup_{k=1}^{N-1}G/G_k$ and $G/G_n$ is at least $R$ in $\Box_\sigma G$. For $n\geq N$, we take $X_n$ maximal in $G/G_n$ such that $B_{R'}(x)$ and $B_{R'}(y)$ are disjoint for every $x,y\in X_n$($X_n$ can be taken by the Zorn Lemma). We then define
	\[\mathcal{U}=\{\bigcup_{k=1}^{N-1}G/G_k\}\cup \bigcup_{n=N}^\infty \{B_{2R'}(x)\colon x\in X_n\}. \] 
    
	\par The collection $\mathcal{U}$ is a uniformly bounded cover of $\Box_\sigma G$. For any $z\in \Box_\sigma G$, there exists $n_0$ such that $z\in G/G_{n_0}$. If $n_0<N$, then $z\in \bigcup_{k=1}^{N-1}G/G_k$. If $n_0\geq N$, as $X_{n_0}$ is maximal, then there exists $x\in X_{n_0}$ such that $B_{R'}(z)\cap B_{R'}(x)\neq\emptyset$. So $z\in B_{2R'}(x)$. Thus, $\mathcal{U}$ is a cover and it is easy to see that the diameters of its members are bounded by $\max\{{\rm diam}(\bigcup_{k=1}^{N-1}G/G_k),s_0\}$. 
    
	\par Finally, we show that $\mathcal{U}$ has $R$-multiplicity at most $K$. For any $z\in \Box_\sigma G$, there exists $n_0$ such that $z\in G/G_{n_0}$. If $n_0<N$, then $B_R(z)$ joints at most one member in $\mathcal{U}$. In the case $n_0\geq N$, if $B_{R}(z)\cap B_{2R'}(x)\neq\emptyset$, then $x\in B_{R+2R'}(z)\subseteq B_{3R'}(z)$. Now consider $B_{3R'}(z)\cap X_{n_0}$. Since $\pi_{n_0}|_{B_{s_0}(1_G)}$ is an isometry, 
	\[|B_{3R'}(z)\cap X_{n_0}|\leq \frac{\nu_{n_0}(B_{4R'}(z))}{\nu_{n_0}(B_{R'}(z))}=\frac{\nu_{n_0}(B_{4R'}(1_{G/G_{n_0}}))}{\nu_{n_0}(B_{R'}(1_{G/G_{n_0}}))}=\frac{m(B_{4R'}(1_G))}{m(B_{R'}(1_G))}\leq K. \]
	Therefore the $R$-multiplicity of $\mathcal{U}$ is at most $K$. So, ${\rm asdim}(\Box_\sigma G)\leq 4^d$.
\end{proof}

\begin{remark}
    Adapting the proof of \cite{DT18} to the case of residually compact groups, one may show that asymptotic dimension of the box space associated to a normal approximation sequence is either infinite or equal to the asymptotic dimension of $G$.  We shall not delve into the details here as we do not need such a sharp bound. 
\end{remark}
	
	\section{Rokhlin dimension for actions of residually compact groups}\label{sec5}
    
	In this section, we introduce the definition of the Rokhlin dimension for the actions of residually compact groups, it is motivated by the definitions for actions of residually finite groups and the reals as presented in \cite[Definition~5.4]{SWZ19} and \cite[Definition~3.6 and Remark~3.10]{HSWW17}, both in turn motivated by \cite{HWZ15}. Also, our definition follows in the same spirit as that of \cite[Definition~4.1]{Sza19}.
    
	\begin{definition}[{\cite[Definition~5.1]{HSWW17}} and {\cite[Definition~4.1]{Sza19}}]\label{def main}
		Let $G$ be a second countable, locally compact group, let $A$ be a C*-algebra, and let $\alpha\colon  G\to {\rm Aut}(A)$ be a point-norm continuous action. Let $H\leq G$ be a closed cocompact subgroup. Let $d\in\mathbb{N}$ be a natural number. 
        
		\par We say that $\alpha$ has \emph{Rokhlin dimension at most $d$ relative to $H$}, and write ${\rm dim}_{\rm Rok}(\alpha,H)\leq d$, if for all separable, $\alpha$-invariant C*-subalgebras $D\subseteq A$, there exist equivariant c.p.c.\ order zero maps
		\[\varphi_l\colon (C(G/H), G\text{-}{\rm shift})\to (F^{(\alpha)}(D,A_{\infty}),\tilde{\alpha}_{\infty})\quad (l=0,\ldots,d) \]
		with $\varphi_0(1)+\cdots+\varphi_d(1)=1$. The value ${\rm dim}_{\rm Rok}(\alpha,H) $ is defined to be the smallest $d\in\mathbb{N}$ such that ${\rm dim}_{\rm Rok}(\alpha,H)\leq d$, or $\infty$ if no such $d$ exists.
        
        \par 
		We say that $\alpha$ has \emph{Rokhlin dimension at most $d$ with commuting towers and relative to $H$} , and write ${\rm dim}_{\rm Rok}^{\rm c}(\alpha,H)\leq d$, if for all separable, $\alpha$-invariant C*-subalgebras $D\subseteq A$, there exist equivariant c.p.c.\ order zero maps
		\[\varphi_l\colon (C(G/H), G\text{-}{\rm shift})\to (F^{(\alpha)}(D,A_{\infty}),\tilde{\alpha}_{\infty})\quad (l=0,\ldots,d) \]
		with pairwise commuting ranges and $\varphi_0(1)+\cdots+\varphi_d(1)=1$. 
	\end{definition}
    
	\begin{remark}
		If $A$ is separable, ${\rm dim}_{\rm Rok}(\alpha,H)\leq d$ if and only if there exist equivariant c.p.c.\ order zero maps $\varphi_0,\ldots,\varphi_d\colon C(G/H)\to F^{(\alpha)}_\infty(A)$ with  $\varphi_0(1)+\cdots+\varphi_d(1)=1$.
	\end{remark}
	
	We also provide an equivalent reformulation. 
    
	\begin{proposition}\label{prop5.1}
		Let $A$ be a separable C*-algebra. Let $G$ be a second countable locally compact group and let $\alpha\colon G\to {\rm Aut}(A)$ be a continuous action. Let $H\leq G$ be a closed cocompact subgroup. Let $d\in\mathbb{N}$ be a natural number. Denote $G\text{-}{\rm shift}$ on $C(G/H)$ by $\lambda$. Then the following statements are equivalent.
		\begin{enumerate}
			\item[(1)] ${\rm dim}_{\rm Rok}(\alpha,H)\leq d$.
			\item[(2)] There exist c.p.c.\ maps $\psi_0,\ldots,\psi_d\colon C(G/H)\to A_{\infty}^{(\alpha)}\cap A'$ satisfying:
			\begin{enumerate}[label=(2\alph*)]
				\item\label{prop:A infty 1} $(\sum_{l=0}^d\psi_l(1)) a=a$ for all $a\in A$;
                
				\item\label{prop:A infty order 0} $\psi_l(f_1)\psi_l(f_2f_3)-\psi_l(f_1f_2)\psi_l(f_3)\in {\rm Ann}(A,A_{\infty})$ for all $l=0,\ldots,d$ and $f_1,f_2,f_3\in C(G/H)$;
                
				\item\label{prop:A infty equiva} $\alpha_{\infty,g}(\psi_l(f))-\psi_l(\lambda_g(f))\in {\rm Ann}(A,A_{\infty})$ for all $l=0,\ldots,d$ and $g\in G$.
			\end{enumerate}
			\item[(3)] For all $\varepsilon>0$ and finite sets $M\subseteq G$,  $F\subseteq A$ and $S\subseteq C(G/H)$, there exist c.p.c.\ maps $\phi_0,\ldots,\phi_d\colon C(G/H)\to A$ satisfying:
			\begin{enumerate}[label=(3\alph*)]
				\item\label{prop:A 1} $(\sum_{l=0}^d\phi_l(1)) a=_{\varepsilon}a$ for all $a\in F$;
                
				\item\label{prop:A order 0} $\phi_l(f_1)\phi_l(f_2f_3)a=_{\varepsilon}\phi_l(f_1f_2)\phi_l(f_3)a$ for all $l=0,\ldots,d$ , $a\in F$and $f_1,f_2,f_3\in S$;
                
				\item\label{prop:A equiva} $\alpha_{g}(\phi_l(f))a=_{\varepsilon}\phi_l(\lambda_g(f))a$ for all $l=0,\ldots,d$, $g\in M$, $f\in S$ and $a\in F$;
                
				\item\label{prop:A abel} $\phi_l(f)a=_{\varepsilon}a\phi_l(f)$ for all $l=0,\ldots,d$, $a\in F$ and $f\in S$.
			\end{enumerate}
		\end{enumerate}
	\end{proposition}
    
	\begin{proof}
		It is similar to the proof of \cite[Lemma~3.7]{Gar17a} and \cite[Proposition~5.5]{SWZ19}. We recall that $\eta_A$ is the quotient map from $\ell^\infty(\mathbb{N},A)$ to $A_\infty$ and $\kappa_{A}$ is the quotient map from $A_\infty\cap A'$ to $F_\infty (A)$.
        
		\par (1)$\Rightarrow$(2). Since ${\rm dim}_{\rm Rok}(\alpha, H)\leq d$, there exist equivariant c.p.c.\ order zero maps
		\[\varphi_l\colon  (C(G/H),\lambda)\to (F^{(\alpha)}_\infty(A)\quad l=0,\ldots,d \]
		with $\varphi_0(1)+\cdots+\varphi_d(1)=1$. Then by the Choi-Effros lifting theorem, there exist c.p.c.\ maps
		\[\psi_l\colon C(G/H)\to A_\infty^{(\alpha)}\cap A' \]
		such that $\varphi_l=\psi_l\circ\kappa_{A}$ for every $l=0,\ldots,d$. Thus, $\sum_{l=0}^d\varphi_l(a)=1$ implies condition~\ref{prop:A infty 1}, every $\varphi_l$ is order zero implies condition~\ref{prop:A infty order 0}, and every $\varphi$ is equivariant implies condition~\ref{prop:A infty equiva}.
        
		\par (2)$\Rightarrow$(1). If there exist c.p.c.\ maps $\psi_0,\ldots,\psi_d\colon C(G/H)\to A_{\infty}^{(\alpha)}\cap A'$ that satisfy conditions~\ref{prop:A infty 1}, \ref{prop:A infty order 0}, and \ref{prop:A infty equiva}. Then define c.p.c.\ maps
		\[\varphi_l=\kappa_{A}\circ \psi_l\quad l=0,\ldots,d .\]
		Then condition~\ref{prop:A infty 1} implies that $\sum_{l=0}^d\varphi_l(1)=1$, condition~\ref{prop:A infty order 0} implies $\varphi_l$ is order zero and condition~\ref{prop:A infty equiva} implies $\varphi_l$ is equivariant.
        
		\par (2)$\Rightarrow$(3). Let $\varepsilon>0$, and let $M\subseteq G$,  $F\subseteq A$ and $S\subseteq C(G/H)$ be finite sets. There exist c.p.c.\ maps $\psi_0,\ldots,\psi_d\colon C(G/H)\to A_{\infty}^{(\alpha)}\cap A'$ satisfying conditions~\ref{prop:A infty 1}, \ref{prop:A infty order 0}, and \ref{prop:A infty equiva}. Applying the Choi-Effros lifting theorem again, we obtain c.p.c.\ maps 
		\[\Phi_l=(\phi_l^n)_{n\in\mathbb{N}}\colon C(G/H)\to \ell^\infty(\mathbb{N},A) \]
		satisfying $\psi_l=\eta_A\circ\Phi_l$ for every $l=0,\ldots,d$. Then conditions~\ref{prop:A infty 1}, \ref{prop:A infty order 0}, and \ref{prop:A infty equiva}. translate to:
        
		\begin{itemize}
			\item $\|\sum_{l=0}^d\phi_l^n(1)a-a\|\stackrel{n\to\infty}\longrightarrow 0$ for all $a\in A$;
			\item $\|\phi_l^n(f_1)\phi_l^n(f_2f_3)-\phi_l^n(f_1f_2)\phi_l^n(f_3)a\|\stackrel{n\to\infty}\longrightarrow 0$  for all $l=0,\ldots,d$ , $a\in A$ and $f_1,f_2,f_3\in C(G/H)$;
			\item $\|\alpha_{g}(\phi_l^n(f))a-\phi_l^n(\lambda_g(f))a\|\stackrel{n\to\infty}\longrightarrow 0$ for all $l=0,\ldots,d$, $g\in G$, $f\in C(G/H)$ and $a\in A$. 
		\end{itemize}
        
		In fact $\psi_l$ maps $C(G/H)$ to $A_\infty\cap A'$, so we also have
        
		\begin{itemize}
			\item $\|\phi_l^n(f)a-a\phi_l^n(f)\|\stackrel{n\to\infty}\longrightarrow0$ for all $l=0,\ldots,d$, $a\in A$ and $f\in C(G/H)$.
		\end{itemize}
        
		Then since $F, M,S$ are all finite sets, we can choose a large enough number $n$ such that $\phi^n_l$ satisfies conditions~\ref{prop:A 1}, \ref{prop:A order 0}, \ref{prop:A equiva}, and \ref{prop:A abel}.
        
		\par (3)$\Rightarrow$(2). Since $ A$ is separable, there is an increasing sequence $(F_n)_{n\in\mathbb{N}}$ of finite subsets of $A$ such that $\overline{\bigcup_{n\in\mathbb{N}}F_n}=A$.  Since $G$ is second countable, it is separable, there is an increasing sequence $(M_n)_{n\in\mathbb{N}}$ of finite subsets of $G$ such that $\overline{\bigcup_{n\in\mathbb{N}}M_n}=G$. Since $G/H$ is a compact metric space, $C(G/H)$ is also separable, there is also an increasing sequence $(S_n)_{n\in\mathbb{N}}$ of finite subsets of $C(G/H)$ such that $\overline{\bigcup_{n\in\mathbb{N}}S_n}=C(G/H)$. Now for every $(\frac{1}{n}, F_n,M_n,S_n)$, there exist c.p.c.\ maps
		\[ \phi_0^n,\ldots,\phi_d^n\colon C(G/H)\to A\]
		satisfying conditions~\ref{prop:A 1}, \ref{prop:A order 0}, \ref{prop:A equiva}, and \ref{prop:A abel}. Then, we have:
        
		\begin{itemize}
			\item $\|\sum_{l=0}^d\phi_l^n(1)a-a\|\stackrel{n\to\infty}\longrightarrow 0$ for all $a\in \bigcup_{k\in\mathbb{N}}F_k$;
			\item $\|\phi_l^n(f_1)\phi_l^n(f_2f_3)-\phi_l^n(f_1f_2)\phi_l^n(f_3)a\|\stackrel{n\to\infty}\longrightarrow 0$  for all $l=0,\ldots,d$ , $a\in\bigcup_{k\in\mathbb{N}}F_k$ and $f_1,f_2,f_3\in\bigcup_{k\in\mathbb{N}}S_k$;
			\item $\|\alpha_{g}(\phi_l^n(f))a-\phi_l^n(\lambda_g(f))a\|\stackrel{n\to\infty}\longrightarrow 0$ for all $l=0,\ldots,d$, $g\in \bigcup_{k\in\mathbb{N}}M_k$, $f\in\bigcup_{k\in\mathbb{N}}S_k$ and $a\in \bigcup_{k\in\mathbb{N}}F_k$.;
			\item  $\|\phi_l^n(f)a-a\phi_l^n(f)\|\stackrel{n\to\infty}\longrightarrow 0$ for all $l=0,\ldots,d$, $a\in \bigcup_{k\in\mathbb{N}}F_k$ and $f\in \bigcup_{k\in\mathbb{N}}S_k$.
		\end{itemize}
        
		\par For every $l=0,\ldots,d$, let $\Phi_l=(\phi_l^n)_{n\in\mathbb{N}}$, $\Phi_l$ is a c.p.c.\ map from $C(G/H)$ to $\ell^\infty(\mathbb{N},A)$. In fact, the image of $\Phi_l$ is contained in $\ell^{\infty,(\alpha)}(\mathbb{N},A)$. Define $\psi_l=\eta_A\circ \Phi_l$ for each $l=0,\ldots,d$. Since $\bigcup_{k\in\mathbb{N}}F_k\subseteq D$, $\bigcup_{k\in\mathbb{N}}M_k\subseteq G$ and $\bigcup_{k\in\mathbb{N}}S_k\subseteq C(G/H)$ are dense, respectively. Hence,  $\psi_0,\ldots,\psi_d$ are c.p.c.\ maps form $C(G/H)$ to $A_{\infty}^{(\alpha)}\cap A'$ that satisfy conditions~conditions~\ref{prop:A infty 1}, \ref{prop:A infty order 0}, and \ref{prop:A infty equiva}.
	\end{proof}
    
	\begin{remark}
    With the above notion, the following are equivalent.
    \begin{enumerate}
        \item[(1)] ${\rm dim}_{\rm Rok}^{\rm c}(\alpha,H)\leq d$.
        \item[(2)] There exist c.p.c.\ maps $\psi_0,\ldots,\psi_d\colon C(G/H)\to A_{\infty}^{(\alpha)}\cap A'$ satisfying: \ref{prop:A infty 1}, \ref{prop:A infty order 0},\ref{prop:A infty equiva} and
        $$(2d)\,\psi_l(f_1)\psi_k(f_2)=\psi_k(f_2)\psi_l(f_1)\quad \text{ for all } \,l,k=0,\ldots,d, l\neq k ,\,f_1,f_2\in C(G/H) \; .$$
        \item[(3)] For all $\varepsilon>0$ and finite sets $M\subseteq G$,  $F\subseteq A$ and $S\subseteq C(G/H)$, there exist c.p.c.\ maps $\phi_0,\ldots,\phi_d\colon C(G/H)\to A$ satisfying: \ref{prop:A 1}, \ref{prop:A order 0}, \ref{prop:A equiva}, \ref{prop:A abel} and
        $$(3e)\,\phi_l(f_1)\phi_k(f_2)a=_{\varepsilon}\phi_k(f_2)\phi_l(f_1)a \quad\text{ for all }\, l,k=0,\ldots,d, l\neq k,a\in F,f_1,f_2\in S.$$
    \end{enumerate}
	\end{remark}
    
	\begin{lemma}
		Let $A$ be a C*-algebra. Let $G$ be a second countable locally compact group and let $\alpha\colon G\to {\rm Aut}(A)$ be a continuous action. If $H_2\leq H_1\leq G$ are two closed cocompact subgroups, then
		\[{\rm dim}_{\rm Rok}(\alpha,H_1)\leq {\rm dim}_{\rm Rok}(\alpha,H_2) \]and
		\[{\rm dim}_{\rm Rok}^{\rm c}(\alpha,H_1)\leq {\rm dim}_{\rm Rok}^{\rm c}(\alpha,H_2).\]
	\end{lemma}
    
	\begin{proof}
		There is an equivariant unital $\ast$-homomorphism 
		\[(C(G/H_1),G\text{-}{\rm shift})\to (C(G/H_2),G\text{-}{\rm shift}). \]
		Therefore, it is easy to obtain ${\rm dim}_{\rm Rok}(\alpha,H_1)\leq {\rm dim}_{\rm Rok}(\alpha,H_2)$ from the definition.
	\end{proof}

	\begin{definition}[see also {\cite[Definition~4.1]{Sza19}}]
		Let $A$ be a C*-algebra, $G$ a second countable, residually compact group and  $\alpha\colon G\to {\rm Aut}(A)$ be a continuous action. Let $\sigma=(G_n)_{n\in\mathbb{N}}$ be a regular approximation. 
        
		\par We define the \emph{Rokhlin dimension of $\alpha$ along $\sigma$} as 
		\[{\rm dim}_{\rm Rok}(\alpha,\sigma)=\sup_{n\in\mathbb{N}}{\rm dim}_{\rm Rok}(\alpha,G_n). \]
		And the \emph{Rokhlin dimension of $\alpha$} as
		\[{\rm dim}_{\rm Rok}(\alpha)=\sup\{{\rm dim}_{\rm Rok}(\alpha,H)\colon H\leq G, H \ \text{is closed cocompact}\}. \]
        
		\par We define the \emph{Rokhlin dimension with commuting towers of $\alpha$ along $\sigma$} as 
		\[{\rm dim}_{\rm Rok}^{\rm c}(\alpha,\sigma)=\sup_{n\in\mathbb{N}}{\rm dim}_{\rm Rok}^{\rm c}(\alpha,G_n). \]
		And the \emph{Rokhlin dimension with commuting towers of $\alpha$} as
		\[{\rm dim}_{\rm Rok}^{\rm c}(\alpha)=\sup\{{\rm dim}_{\rm Rok}^{\rm c}(\alpha,H)\colon H\leq G, H \ \text{is closed cocompact}\}. \]
	\end{definition}
    
	\begin{remark}
		If $G$ is a residually finite group, then the definition coincides with the Rokhlin dimension of residually finite groups as defined in \cite[Definition~5.4 and Definition~5.8]{SWZ19}.
        
 		\par If $G$ is chosen as $\mathbb{R}$, then the definition coincides with the Rokhlin dimension of $\mathbb{R}$ as defined in \cite[Definition~3.6 and Remark~3.10]{HSWW17}.
	\end{remark}
    
	The following is a proposition regarding the Rokhlin dimension for actions of residually compact groups. It is an analogy to \cite[Theorem~3.8]{Gar17a}.
    
	\begin{proposition}
		Let $A$ be a separable C*-algebra, $G$ be a second countable, residually compact group, and let $\alpha\colon G\to {\rm Aut}(A)$ be a continuous action of $G$ on $A$. Let $I$ be an $\alpha$-invariant ideal in $A$, and denote by $\bar{\alpha}\colon G\to {\rm Aut}(A/I)$ the induced action on the quotient. Then we have:
		\[{\rm dim}_{\rm Rok}(\bar{\alpha})\leq {\rm dim}_{\rm Rok}(\alpha), \]
		\[{\rm dim}_{\rm Rok}^{\rm c}(\bar{\alpha})\leq {\rm dim}_{\rm Rok}^{\rm c}(\alpha). \]
		
	\end{proposition} 
    
	\begin{proof}
		\par  Let $H$ be a closed cocompct subgroup of $G$, and let $d={\rm dim}_{\rm Rok}(\alpha,H)$. Let $\varepsilon>0$ and $M\subseteq G,F\subseteq A/I,S\subseteq C(G/H)$ be finite sets. Choose $F'\subseteq A$ such that $\pi(F')=F$ where $\pi$ is the quotient map from $A$ to $A/I$. Applying Proposition~\ref{prop5.1} for $\varepsilon,F',M,S$, there exist  c.p.c.\ maps
		\[\psi_0,\ldots\psi_d\colon C(G/H)\to A \]
		satisfying conditions in Proposition~\ref{prop5.1}. Then define $$\varphi_l=\pi\circ\psi$$ for each $l=0,\ldots,d$. These maps $\varphi_0,\ldots,\varphi_d$ also satisfy the conditions in Proposition~\ref{prop5.1} for $(\varepsilon,F,M,S)$. Thus, ${\rm dim}_{\rm Rok}(\bar{\alpha},H)\leq {\rm dim}_{\rm Rok}(\alpha,H) $. Then $\sup_{H}{\rm dim}_{\rm Rok}(\bar{\alpha},H)\leq\sup_{H} {\rm dim}_{\rm Rok}(\alpha,H)$. That is, ${\rm dim}_{\rm Rok}(\bar{\alpha})\leq {\rm dim}_{\rm Rok}(\alpha)$. The second inequality is the same.
	\end{proof}

	\section{Nuclear dimension of crossed products}\label{sec6}
    
	In this section, we consider the nuclear dimension of the crossed product C*-algebras. 
	We first recall the definition of the nuclear dimension and a technical characterization of it that we will use later.
    
	\begin{definition}[{\cite[Definition~2.1]{WZ10}}]
		Let $A$ be a C*-algebra and $n\in \mathbb{N}$ be a natural number. The \emph{nuclear dimension} of $A$ is the least integer $n$ such that the following holds:
		\par For any finite set $F\subseteq A$ and $\varepsilon>0$, there exists a finite-dimensional C*-algebra $\mathcal{F}=\mathcal{F}^0\oplus \dots \oplus \mathcal{F}^n$, a c.p.c.\ map $\psi \colon A\to \mathcal{F}$ and a c.p.\ map $\varphi\colon \mathcal{F}\to A$ such that the restriction of $\varphi$ to $\mathcal{F}^i$ is c.p.c.\ order zero for every $i=0,\ldots,n$ and
		\[\|\varphi(\psi(a))-a\|\leq \varepsilon \]
		for all $a\in F$. We denote the nuclear dimension of $A$ by ${\rm dim}_{\rm nuc}(A)$. If no such $n$ exists, we write ${\rm dim}_{\rm nuc}(A)=\infty$.
	\end{definition}
	
	\begin{lemma}[{\cite[Lemma 2.14]{HSWW17}}]\label{lem6.1}
		Let $A$ be a C*-algebra, and let $d,n>0$. Denote by $\iota\colon A\to A_\infty$ the canonical inclusion as constant sequences. Suppose that for every finite set $F\subseteq A$ and for any $\varepsilon>0$, there exists a C*-algebra $B=B_{F,\varepsilon}$ with ${\rm dim}_{\rm nuc}(B)\leq d$. a c.p.c.\ map $\varphi\colon A\to B$ and a family of c.p.c.\ order zero maps $\psi^0,\ldots,\psi^n \colon  B\to A_\infty$ such that
		\[\|\iota (x)-\sum_{j=0}^n \psi^j(\varphi(x))\|\leq \varepsilon \]
		for all $x\in \mathcal{F}$. Then
		\[{\rm dim}_{\rm nuc}^{+1}(A)\leq (d+1)(n+1). \]
	\end{lemma}
	
	\begin{theorem}\label{thm:dimnuc}
		Let $A$ be a separable C*-algebra, G be a second countable, residually compact group and $\alpha\colon G\to {\rm Aut}(A)$ a continuous action. Let $\sigma=(G_n)_{n\in \mathbb{N}}$ be a regular approximation of $G$. Then the following inequality holds:
		\[{\rm dim}_{\rm nuc}^{+1}(A\rtimes_\alpha G)\leq {\rm asdim}^{+1}(\Box_{\sigma}G)\cdot {\rm dim}_{\rm Rok}^{+1}(\alpha,\sigma)\cdot {\rm dim}_{\rm nuc}^{+1}(A). \]
	\end{theorem}
    
	\begin{proof}
		First of all, if one of the dimensions on the right side of the ineuqation is infinity, there is nothing to prove. Then we assume that the asymptotic dimension of the box space $\Box_\sigma G$, the Rokhlin dimension of $\alpha$ along to $\sigma$, and the nuclear dimension of $A$ are all finite. By Corollary~\ref{cor4.15}, $G$ is amenable. Let $l={\rm asdim}(\Box_\sigma G)$, $d={\rm dim}_{\rm Rok}(\alpha,\sigma)$. For any finite set $F\subseteq A\rtimes_\alpha G$ and any $\varepsilon>0$. Without loss of generality, we may assume that $F\subseteq C_c(G,A)$. Then there is a compact set $M\subseteq G$ such that ${\rm supp}(x)\subseteq M$ for every $x\in F$. Set $L=\sup_{x\in F} \|x\|_{L^1}$.
        
		\par Since the square function $\centerdot^{\frac{1}{2}}\colon [0,1]\to [0,1]$ is uniformly continuous, for $\frac{\varepsilon}{(l+1)(d+1)L}>0$ there is a $\delta>0$ such that 
		\[\|a^\frac{1}{2}-b^\frac{1}{2}\|\leq \frac{\varepsilon}{(l+1)(d+1)L}\quad \text{whenever} \|a-b\|\leq \delta, a,b\in [0,1]. \]
        
		\par For compact set $M$ and $\delta>0$, applying the Lemma~\ref{lem main} there exist natural number $n$ and compactly supported functions $\mu_i\colon G\to [0,1]$ for all $i=0,\ldots,l$ satisfying:
		\begin{enumerate}
			\item for every $i=0,\ldots,l$, one has 
			\[{\rm supp}(\mu_i)\cap {\rm supp}(\mu_i)h=\emptyset \quad\text{ for all }\quad h\in G_n\backslash\{1_{G_n}\};\]
			\item for every $g\in G$, one has
			\[\sum_{i=1}^l\sum_{h\in G_n}\mu_i(gh)=1; \]
			\item for every $i=0,\ldots,l$ and $t\in M$, one has\[\|\mu_i-\mu_i(t^{-1}\cdot-)\|_{\infty}\leq \delta. \]
		\end{enumerate}
		Therefore, we have $\|\mu_i^\frac{1}{2}-\mu_i^\frac{1}{2}(t^{-1}\cdot-)\|_{\infty}\leq \frac{\varepsilon}{(l+1)(d+1)L} $ for all $t\in M$.
        
		\par Let $\lambda\colon G\to {\rm Aut}(G/G_n)$ be the shift $\lambda_t(f)(sG_n)=f((t^{-1}s)G_n)$ where $s,t\in G,f\in C(G/G_n)$. Then we have the induced action $\lambda\otimes\alpha\colon G\to {\rm Aut}(C(G/G_n)\otimes A)$ ($C(G/G_n)\otimes A\cong C(G/G_n,A)$). Since ${\rm dim}_{\rm Rok}(\alpha,\sigma)=d$, there exist $\lambda$-$\alpha_\infty$ equivariant c.p.c.\ order zero maps 
		\[\varphi_0,\ldots,\varphi_d\colon  C(G/G_n)\to F_{\infty}^{(\alpha)} (A) \]
		with\[\varphi_0(1)+\cdots+\varphi_d(1)=1. \]
		Then by Remark~\ref{rem2.21} the maps $\varphi_0,\ldots,\varphi_d$ induce  $\lambda\otimes\alpha$-$\alpha_\infty$ equivariant c.p.c.\ order zero maps
		\[\eta_0,\ldots,\eta_d\colon C(G/G_n)\otimes A\to A_\infty^{(\alpha)} \]
		with \[\eta_0(1\otimes a)+\cdots+\eta_d(1\otimes a)=a \] for all $a\in A$.
		For each $j\in\{0,\ldots,d\}$, applying Lemma~\ref{lem2.31}, the map $\eta_j$ induce  a c.p.c.\ order zero map
		\[\eta_j\rtimes G\colon (C(G/G_n)\otimes A)\rtimes_{\lambda\otimes\alpha G}\to A_\infty^{(\alpha)} \rtimes_{\alpha_\infty}G. \]
		And we have the equation that 
		\[\sum_{j=0}^d(\eta_j\rtimes G)\circ ((1\otimes {\rm id}_A)\rtimes G)=\left(\sum_{j=0}^d\eta_j\circ(1\otimes {\rm id}_A)\right)\rtimes G=\iota_A\rtimes G.\]
        
		\par  Let $\beta \colon G\to {\rm Aut}(C_0(G))$ be the shift, $\beta_t(f)(s)=f(t^{-1}s) $ where $ s,t\in G,f\in C_o(G)$, we also have the induced action $\beta\otimes \alpha\colon G\to {\rm Aut}(C_0(G)\otimes A)$. For every $\mu_i,i=0,\ldots,l$, denote $U^i=\{g\in G\colon \mu_i(g)>0\}$, consider the set
		\[B^i=\{f\in L^1(G,C_0(G,A))\colon f(t)(x)=0 \quad\text{whenever}\quad t^{-1}x\notin U^i \text{or} \,x\notin U^i \} .\]
		Since the canonical quotient map $\pi_n\colon G\to G/G_n$ restricts to $U^i$ is injective, we can view every element in $B^i$ as an element in $L^1(G,C(G/G_n,A))$. 
		
		\par We claim that $B^i$ is a $\ast$-subalgebra of $L^1(G,C_0(G,A))$ and $L^1(G,C(G/G_n,A))$, and the two restricted operations are consistent. Let $f,g\in B^i$, consider the product in $L^1(G,G_0(G,A))$, 
		\begin{align*}
			f\ast g(t)(s)&=\int_G f(r)(s)(\beta_r\otimes\alpha_r)(g(r^{-1}t))(s)dr\\
			&=\int_G f(r)(s)\alpha_r(g(r^{-1}t)(r^{-1}s))dr.
		\end{align*}
		If $s\notin U^i$ then $f(r)(s)=0$. If $t^{-1}s\notin U^i$ then $(r^{-1}t)^{-1}(r^{-1}s)\notin U^i$, so $g(t-r)(s-r)=0$. Thus, $f\ast g\in B^i$.
        
		\par If we consider the product in $L^1(G,G(G/G_n,A))$, 
		\begin{align*}
			f\ast g(t)(\bar{s})&=\int_G f(r)(\bar{s})(\lambda_r\otimes\alpha_r)(g(r^{-1}t))(\bar{s})dr\\
			&=\int_G f(r)(\bar{s})\alpha_r(g(r^{-1}t)(\overline{r^{-1}s}))dr.
		\end{align*}
		If $\bar{s}\notin \pi_n(U^i)$ then $f(r)(\bar{s})=0$, so we set $s\in U^i$. But if $r^{-1}s\notin U^i$ we also have $f(r)(\bar{s})=0$ hence the two expressions are the same. The convolution is similar to the product. Then the claim is proved.
		
		\par We denote by $B^i_\beta$ and $B^i_\lambda$ the completions of $B^i$ in $(C_0(G)\otimes A)\rtimes_{\beta\otimes\alpha }G$ and $(C(G/G_n)\otimes A)\rtimes_{\lambda\otimes\alpha}G$. Using the left regular representation, we can represent those two crossed product C*-algebras to the same $B(H)$ where $H=L^2(G\times G,H')$ and $H'$ is the Hilbert space where $A$ acts on. Thus, the norm of $f\in B^i$ is the same in either completion. Hence, the identity map $B^i\to B^i$ extends to an isomorphism
		\[\zeta_i\colon B^i_\beta\to B^i_\lambda. \]
        
		\par Viewing $\mu_i$ as an element of $\mathcal{M}((C_0(G)\otimes A)\rtimes_{\beta\otimes\alpha}G)$. Let \[B_{\mu_i}=\overline{\mu_i((C_0(G)\otimes A)\rtimes_{\beta\otimes\alpha}G)\mu_i}.\]
		$B_{\mu_i}$ is a hereditary subalgebra of $(C_0(G)\otimes A)\rtimes_{\beta\otimes\alpha}G$. \par We claim that $B_{\mu_i}=B^i_\beta$. If $f\in C_c(G,C_0(G,A))$, by the canonical embedding, we have $$(\mu_if\mu_i)(t)(s)=f(t)(s)\mu_i(s)\mu_i(t^{-1}s),t,s\in G.$$
		If $s\notin U^i$ then $\mu_i(s)=0$, if  $t^{-1}s\notin U^i$ then $\mu_i(s-t)=0$. Thus, $\mu_if\mu_i\in B^i$ and $B_{\mu_i}\subseteq B^i_\beta$. On the other hand, $\mu_iC_c(G,C_0(G,A))\mu_i$ is dense in $B^i$. Thus, the claim is proved.
        
		\par For each $i=0,\ldots,l$, $\mu_i$ can be regarded as an element in $C(G/G_n)$, we denote it by $\overline{\mu_i}$. Viewing $\overline{\mu_i}$ as an element of $\mathcal{M}((C(G/G_n)\otimes A)\rtimes_{\lambda\otimes\alpha}G)$. We can also let \[C^{i}=\overline{\overline{\mu_i}((C(G/G_n)\otimes A)\rtimes_{\lambda\otimes\alpha}G)\overline{\mu_i}}.\]
		$C^{i}$ is a hereditary subalgebra of $(C(G/G_n)\otimes A)\rtimes_{\lambda\otimes\alpha}G$. Similarly, $C^i$ is the completion of the set 
			\[C^i_0=\{f\in L^1(G,C(G/G_n,A))\colon f(t)(\overline{x})=0 \quad\text{whenever}\quad \overline{t^{-1}x}\notin \pi_n(U^i) \text{or} \,\overline{x}\notin \pi_n(U^i) \} .\]
		In fact, $B^i_\lambda$ is a subalgebra of $C^i$. 
		
		\par Since there is a canonical equivariant embedding from $A$ to $C(G/G_n)\otimes A$, we can definite  c.p.c.\ maps 
		\[\phi_i\colon  A\rtimes_\alpha G\to C^i \quad \text{ via}\quad  \phi_i(a)=\overline{\mu_i}^{\frac{1}{2}}a\overline{\mu_i}^{\frac{1}{2}}, \]
		for each $i=0,\ldots,l$. Thus, we can define a c.p.c.\ map
		\[\Phi\colon   A\rtimes_\alpha G\to C^0\oplus\cdots\oplus C^l \quad\text{via}\quad \Phi(a)=(\phi_0(a),\ldots,\phi_l(a)). \]
        
		\par Then, for each $i=0,\ldots,l$, we define a map $\phi'_i$ from $C^i$ to $B_\beta^i$ which is induced by 
		\[\phi'_i(f)(t)(s)=\left\{
		\begin{array}{rl}
			f(t)(\bar{s}) & \text{if}\, s\in U^i \, \text{and}\, t^{-1}s\in U^i,\\
			0 & \text{otherwise},
		\end{array}
		\right. \]
		where $f\in C^i_0$. The map $\phi'_i$ is a projection. Moreover, by the regular representation, we can represent both $f$ and $\phi'_i(f)$ in $B(L^2(G\times G,H'))$ and show that $\phi'_i$ is contractive. Thus, by a theorem of Tomiyama(\cite{Tom59}), we can prove that each $\phi'_i$ is a conditional expectation. We can define a c.p.c.\ map 
		\[\Phi'=(\phi'_0,\ldots,\phi'_l)\colon C^0\oplus\cdots\oplus C^l\to B^0_\beta \oplus\cdots\oplus B^l_\beta. \]
		
		\par Let $\kappa\colon A\rtimes_\alpha G\to A^\alpha_\infty \rtimes_{\alpha_\infty}G$ be the natural inclusion that is induced by the equivariant inclusion of $A$ into $A^\alpha_\infty$ as constant sequences. We define a homomorphsim $$P_i\colon B^0_\beta \oplus\cdots\oplus B^l_\beta \to B^i_\beta ,\quad\text{via}\quad P_i(f_0\oplus\cdots\oplus f_l)=f_i,$$ 
		where $i=0,\ldots,l,f_i\in B^i_\beta$. Now consider the following diagram: \par

        \begin{adjustbox}{center}
            \begin{tikzpicture}
			\node (A) at (0,0) {$ A\rtimes_\alpha G$};
			\node (B) at (3,-3) {$C^0\oplus\cdots\oplus C^l$};
			
			\node (D) at (6,-6) {$B^0_\beta \oplus\cdots\oplus B^l_\beta$};
			\node (E) at (8,-4) {$B^i_\beta$};
			\node (F) at (10,-2) {$B^i_\lambda$$\subseteq C^i$};
			\node (G) at (12,0) {$A^\alpha_\infty \rtimes_{\alpha_\infty}G$};
			\draw [-to](A)-- node[above] {$\kappa$} (G);
			\draw [-to](A)--node[below left]{$\Phi$}(B);
			\draw [-to](B)--node[below left]{$\Phi'$}(D);
			
			\draw [-to](D)--node[below right]{$P_i$}(E);
			\draw [-to](E)--node[below right]{$\zeta_i$}(F);
			\draw [-to](F)--node[below right]{$\eta_j\rtimes G $}(G);
			
		\end{tikzpicture}
        \end{adjustbox}
		
	  Where $j=0,\ldots,d,i=0,\ldots,l$ . We wish to show that 
		\[\|\sum_{j=0}^d\sum_{i=0}^l(\eta_j\rtimes G) \circ \zeta_i\circ P_i\circ \Phi'\circ\Phi (x)-\kappa(x)\|\leq\varepsilon. \]
		for all $x\in F$.
        
		\par 
		For each $x\in F$, it can be regarded as an element in $C_c(G,C(G/G_n,A))$. Thus, for every $i=0,\ldots,l$, $\overline{\mu_i}x$ is an element in $C_c(G,C(G/G_n,A))$,
		$$\overline{\mu_i}x(t)=\overline{\mu_i}\otimes x(t),\quad \overline{\mu_i}x(t)(\bar{s})=\mu_i(\bar{s})x(t),\quad t\in G, \bar{s}\in G/G_n.$$
		Furthermore, $$\zeta_i\circ P_i\circ \Phi(a)(t)(\bar{s})=x(t)\mu_i^\frac{1}{2}(s)\mu_i^\frac{1}{2}(t^{-1}s)$$
		where $s\in U^i$ and $t^{-1}s\in U^i$ and it is zero when otherwise. Thus
        
		\begin{align*}
			&\|\zeta_i\circ P_i\circ \Phi'\circ\Phi(x)(t)(\bar{s})- \overline{\mu_i}x(t)(\bar{s})\|\\
            &=\|x(t)\mu_i^\frac{1}{2}(s)(\mu_i^\frac{1}{2}(t^{-1}s)-\mu_i^\frac{1}{2}(s))\|\\
			&\leq \|x(t)\|\frac{\varepsilon}{(l+1)(d+1)L},
		\end{align*}
        
		where $t\in G$ and $s\in U^i$. Therefore, we have
        
		\begin{align*}
			&\|\eta_j\rtimes G) \circ \zeta_i\circ P_i\circ\Phi'\circ\Phi(x)-\eta_j\rtimes G)(\overline{\mu_i})x\|\\
            &\leq \| \zeta_i\circ P_i\circ\Phi'\circ\Phi(x)-\overline{\mu_i}x\|\\
			&\leq \|x\|_{L^1}\frac{\varepsilon}{(l+1)(d+1)L}\leq \frac{\varepsilon}{(l+1)(d+1)}.
		\end{align*}
        
		for all $x\in F$ and $j=0,\ldots,d,i=0,\ldots,l$. Hence for every $x\in F$ we have
        
		\begin{align*}
			&\|\sum_{j=0}^d\sum_{i=0}^l(\eta_j\rtimes G) \circ \zeta_i\circ P_i\circ\Phi'\circ \Phi (x)-\kappa(x)\|\\
            &\leq (l+1)(d+1)\frac{\varepsilon}{(l+1)(d+1)}\\&+\|\sum_{j=0}^d\sum_{i=0}^l(\eta_j\rtimes G)(\overline{\mu_i}x)-\kappa(x)\|\\
			&=\varepsilon+\|\sum_{j=0}^d(\eta_j\rtimes G)(x)-\kappa(x)\|=\varepsilon.
		\end{align*}
		
		\par Takai's duality theorem states that $C_0(G,A)\rtimes_{\beta\otimes\alpha} G \cong A\otimes \mathcal{K}(L^2(G))$ and for each $i=0,\ldots,l$, $B_\beta^i$ is a hereditary C*-subalgebra of $C_0(G,A)\rtimes_{\beta\otimes\alpha}G$. Thus, for every $i=0,\ldots,l$
		\[{\rm dim}_{\rm nuc}(B^i_\beta)\leq  {\rm dim}_{\rm nuc}(A), \]
		then \[{\rm dim}_{\rm nuc}( B^0_\beta \oplus\cdots\oplus B^l_\beta)
		\leq  {\rm dim}_{\rm nuc}(A).\]
		From Lemma~\ref{lem2.21} the map $\kappa$ is compatible with the standard embedding from $A\rtimes_\alpha G$ to $(A\rtimes_\alpha G)_\infty$. Then we can apply Lemma~\ref{lem6.1} and obtain that
		$${\rm dim}_{\rm nuc}^{+1}(A\rtimes_\alpha G)\leq {\rm asdim}^{+1}(\Box_{\sigma}G)\cdot {\rm dim}_{\rm Rok}^{+1}(\alpha,\sigma)\cdot {\rm dim}_{\rm nuc}^{+1}(A) \; . $$
	\end{proof}

	\section{$D$-absorption of crossed products}\label{sec7}
    
	\par In this section, we study the permanence with respect to $D$-absorption for a strongly self-absorbing C*-algebra $D$. We first recall the definition. 
    
	\begin{definition}[{\cite[Definition~1.3]{TW07}}]
		A separable, unital C*-algebra $D$ is called \emph{strongly self-absorbing} if there is an isomorphism $\varphi\colon D\to D\otimes D$ that is approximately unitarily equivalent to ${\rm id}_D\otimes 1_D\colon D\to D\otimes D$, that is to say, there is a sequence of unitaries $u_n\in D\otimes D$ such that 
		\[\varphi(x)=\lim_{n\to\infty}u_n(x\otimes 1_d )u_n^\ast \]
		for all $x\in D$.
		\par A C*-algebra $A$ is called \emph{$D$-absorbing} (or \emph{$D$-stable}) if $A\cong A\otimes D$.
	\end{definition}
    
	Now we recall some lemmas that will be used.
    
	\begin{lemma}[{\cite[Lemma~6.5]{HSWW17}}]\label{lemd1}
		Let $A$ be a separable C*-algebra and let $D$ be a strongly self-absorbing C*-algebra. Suppose that $\alpha\colon G\to {\rm Aut}(A)$ is a point-norm continuous action of a second countable, locally compact group. If there exists a unital $\ast$-homomorphism from $D$ to the fixed point algebra $F_\infty^{(\alpha)}(A)^{\tilde{\alpha}_\infty}$, then $A\rtimes_\alpha G$ is $D$-absorbing.
	\end{lemma}
	
	\begin{lemma}[{\cite[Lemma~6.8]{HSWW17}}]\label{lemd2}
		Let $D$ be a strongly self-absorbing C*-algebra. Let $B$ be a unital C*-algebra. Suppose that $\Psi^1,\ldots,\Psi^n\colon D\to B$ are c.p.c.\ order zero maps with pairwise commuting images such that $\Psi^1(1_D)+\cdots+\Psi^n(1_D)=1_B$. Then there exists a unital $\ast$-homomorphism from $D$ to $B$.
	\end{lemma}
	
	\begin{lemma}[{\cite[Lemma~6.9]{HSWW17}}]\label{lem7.4}
		Let $A$ and $B$ be C*-algebras, and let $\phi\colon A\to B$ be a c.p.c.\ order zero map. Then for every $x,y,x',y'\in A$, we have
		\[\|\phi(x)\phi(y)-\phi(x')\phi(y')\|\leq \|xy-x'y'\|. \]
	\end{lemma}
	
	\begin{lemma}[{\cite[Lemma~6.10]{HSWW17}}]\label{lem7.5}
		Let $Y$ be a locally compact Hausdorff space and let $A,B$ be two C*-algebras. Let $\phi_1,\phi_2\colon C_0(Y)\to B$ be two c.p.c.\ order zero maps with commuting images. Then for every two functions $f_1,f_2\in C_0(Y,A)\cong C_0(Y)\otimes A$, we have
		\[ \|[(\phi_1\otimes{\rm id}_A)(f_1),(\phi_2\otimes{\rm id}_A)(f_2)]\|_{B\otimes_{\rm max}A}\leq \max_{y_1,y_2\in Y}\|[f_1(y),f_2(y)] \|.\]
	\end{lemma}
	
	\begin{lemma}[{\cite[Lemma~1.4]{Kas88}}]\label{lem7.6}
		Let $A$ be a $\sigma$-unital C*-algebra, $G$ a $\sigma$-compact, locally compact group and $\alpha\colon G\to {\rm Aut}(A)$ a point-norm continuous action. Then there exists an approximate identity $(e_n)_{n\in\mathbb{N}}$ for $A$ such that $\|\alpha_g(e_n)-e_n\|\to\infty$ uniformly on compact subsets of $G$.
	\end{lemma}
	The following is the main theorem of this section.
	
	\begin{theorem}\label{thm: D-ab}
		Let $D$ be a strongly self-absorbing C*-algebra. Let $A$ be a separable, $D$-absorbing C*-algebra. Let $G$ be a second countable residually compact group with a regular approximation $\sigma=(G_n)_{n\in\mathbb{N}}$ and $\alpha\colon G\to {\rm Aut}(A)$ a continuous action. If ${\rm asdim}(\Box_{\sigma}G)<\infty$ and ${\rm dim}_{\rm Rok}^{\rm c}(\alpha,\sigma)<\infty$. Then $A\rtimes_\alpha G$ is $D$-absorbing.
	\end{theorem}
    
	\begin{proof}
		We assume ${\rm asdim}(\Box_{\sigma}G)=s$ and ${\rm dim}_{\rm Rok}^{\rm c}(\alpha,\sigma)=m$. Let $M$ be a compact subset of $G$ and $\varepsilon>0$. Since ${\rm asdim}(\Box_{\sigma}G)=s<\infty$, by Lemma~\ref{lem main} there exist $n_0$ and compactly supported continuous functions $\mu_0,\ldots,\mu_s\colon G\to [0,1]$ satisfying 
		\begin{itemize}
			\item for every $i=0,\ldots,s$,
			\[{\rm supp}(\mu_i)\cap {\rm supp}(\mu_i)h=\emptyset \quad\text{ for all }\quad h\in G_{n_0}\backslash\{1_{G_{n_0}}\};\]
			\item for every $g\in G$,
			\[\sum_{i=1}^s\sum_{h\in G_{n_0}}\mu_i(gh)=1; \]
			\item for every $i=0,\ldots,s$ and $t\in M$,  \[\|\mu_i-\mu_i(t^{-1}\cdot-)\|_{\infty}\leq \frac{\varepsilon}{s+1}. \]
		\end{itemize}
        
		\par  For each $i=0,\ldots,s$, $\mu_i$ can be regarded as an element in $C(G/G_{n_0})$, we denote it by $\overline{\mu_i}$. We have $\overline{\mu_i}(gG_{n_0})=\sum_{h\in G_{n_0}}\mu_i(gh)$, for all $g\in G$. We also denote $M'=\bigcup_{i=0}^s{\rm supp}(\mu_i)$, it is a compact subset of $G$.
        
		\par Consider the c.p.c.\ order zero maps
		\[\varphi^i\colon A\to C(G/G_{n_0})\otimes A\cong C(G/G_{n_0},A),\quad i=0,\ldots,s  \]
		given by 
		\[\varphi^i(a)(gG_{n_0})=\sum_{h\in G_{n_0}}\mu_i(gh)\alpha_{gh}(a) \]
		for all $g\in G$. These maps are well defined, and we have 
		\begin{equation}\label{7.1}
			\|(\lambda_t\otimes\alpha_t)\circ \varphi^i-\varphi^i\|\leq\varepsilon
		\end{equation}
		
		for all $i=0,\ldots,s$ and $t\in M$, where $\lambda$ is the $G$-shift on $C(G/G_{n_0})$. To see this we have
		\begin{align*}
			&\|(\lambda_t\otimes\alpha_t)(\varphi^i(a))(gG_{n_0})-\varphi^i(a)(gG_{n_0})\|\\
			=&\|\alpha_t(\varphi^i(a)(t^{-1}gG_{n_0}))-\varphi^i(a)(gG_{n_0})\|\\
			=&\|\sum_{h\in G_{n_0}}\mu_i(t^{-1}gh)\alpha_t(\alpha_{t^{-1}gh}(a))-\sum_{h\in G_{n_0}}\mu_i(gh)\alpha_{gh}(a)\|\\
			=& \|\sum_{h\in G_{n_0}}(\mu_i(t^{-1}gh)-\mu_i(gh))\alpha_{gh}(a)\|\\
			\leq& \varepsilon.
		\end{align*}
        
		\par Since $A$ is $D$-absorbing, there exists a unital $\ast$-homomorphism $\tilde{k}\colon D\to F_\infty(A)$. As $D$ is strongly self-absorbing, $D$ is nuclear. Then we can apply the Choi-Effros lifting theorem(\cite{CE76}) and find a c.p.c.\ lift of this map to $\ell^\infty(\mathbb{N},A)$, and represent it by a sequence of c.p.c.\ maps $k_n\colon D\to A$ such that $\tilde{k}(d)$ is the image of $(k_1(d),\ldots,k_n(d),\ldots)$ for all $d\in D$. Thus, we have 
		\begin{enumerate}[label=(a\arabic*)]
			\item\label{D-ab a1} $\|k_n(1)a-a\|\rightarrow 0$ as $ n\to\infty$, for all $a\in A$ .
			\item \label{D-ab a2} $\|(k_n(d_1)k_n(d_2)-k_n(d_1 d_2))a\|\to 0$ as $ n\to\infty$, for all $d_1,d_2\in D$.
			\item\label{D-ab a3} $\|[a,k_n(d)]\|\to 0$ as $ n\to \infty$, for all $a\in A$ and $d\in D$.
		\end{enumerate}
        
		\par Let $F_D\subseteq D$ and $F_A\subseteq A$ be finite subsets, with $1_d\in F_D$. We may assume that they consist of elements with norm at most 1. Applying the lemma~\ref{lem7.6}, we may also assume that $F_A$ contains a positive element $e$ of norm 1, such that $\|ea-a\|<\varepsilon$ and $\|ae-a\|<\varepsilon$ for all $a\in F_A\backslash\{e\}$ and such that $\|\alpha_t(e)-e\|<\varepsilon$ for all $t\in M\cup M'$.
        
		\par By picking $\kappa=\kappa_n$ for some sufficiently large $n$, we have a c.p.c.\ map satisfying:
		\begin{enumerate}[label=(b\arabic*)]
			\item  \label{D-ab b1} $\|k(1)\alpha_t(a)-\alpha_t(a)\|\leq\varepsilon$, for all $a\in F_A$ and $t\in M\cup M'$ ;
			\item \label{D-ab b2} $\|(k_n(d_1)k_n(d_2)-k_n(d_1 d_2))\alpha_t(a)\|\leq\varepsilon$, for all $d_1,d_2\in F_D$ and $t\in M\cup M'$;
			\item\label{D-ab b3} $\|[\alpha_t(a),k_n(d)]\|\leq\varepsilon$, for all $a\in F_A$, $d\in F_D$ and $t\in M\cup M'$.
		\end{enumerate}
        
		\par We choose inductively c.p.c.\ maps $k^{(i,l)}\colon D\to A$ for all $i=0,\ldots,s$ and $l=0,\ldots,m$ satisfying the above conditions and such that we also have
		\begin{enumerate}[label=(b\arabic*)]
        \setcounter{enumi}{3}
			\item \label{D-ab b4} $\|[\alpha_t\circ k^{(i,l)}(d_1),k^{(i',l')}(d_2)]\|\leq\varepsilon$, for all $d_1,d_2\in D$, $t\in M\cup M'$ and $(i,l)\neq(i',l')$.
		\end{enumerate}
        
		\par Combining the properties of the maps $k^{(i,l)}$ and $\varphi^i$, we have the following :
		\begin{enumerate}[label=(c\arabic*)]
			\item \label{D-ab c1} $\|(1-\sum_{i=0}^s\varphi^i\circ k^{(i,l)}(1))\cdot (1_{C(G/G_{n_0})}\otimes a) \|\leq (s+1)\varepsilon $, for all $l=0,\ldots,m$ and $a\in F_A$;
			\item \label{D-ab c2} $\|(\lambda_t\otimes\alpha_t)\circ \varphi^i\circ k^{(i,l)}- \varphi^i\circ k^{(i,l)} \|\leq \varepsilon$, for all $i=0,\ldots,s$, $l=0,\ldots,m$ and $t\in M$.
		\end{enumerate}
		Here, \ref{D-ab c1} follows from \ref{D-ab b1} and \ref{D-ab c2} follows from \eqref{7.1}.
        
		\par Since $m={\rm dim}_{\rm Rok}^{\rm c}(\alpha,\sigma)<\infty$, there exist $\lambda$-$\tilde{\alpha}_\infty$ equivariant c.p.c.\ order zero maps
		\[\phi_0,\ldots,\phi_m\colon C(G/G_{n_0})\to F^{(\alpha)}_\infty(A) \]
		with pairwise commuting images and such that $\phi_0(1)+\cdots+\phi_m(1)=1$. By Remark~\ref{rem2.21}, there is an canonical $\tilde{\alpha}_\infty\otimes\alpha$-$\alpha_\infty$ equivariant $\ast$-homomorphsim\[\theta\colon F^{(\alpha)}_\infty(A)\otimes_{\rm max} A\to A_\infty^{(\alpha)}. \]
		\par For all $i=0,\ldots,s$ and $l=0,\ldots,m$, we define c.p.c.\ maps $\psi^{(i,l)}\colon D\to A^{(\alpha)}_\infty$ as the composition
        
		\par
		\begin{tikzpicture}
			\node (A) at (0,0) {$ D$};
			\node (B) at (1.5,0) {$A$};
			\node (C) at (4,0) {$C(G/G_{n_0})\otimes A$};
			\node (D) at (9,0) {$F^{(\alpha)}_\infty(A)\otimes_{\rm max} A$};
			\node (E) at (12,0) {$A_\infty^{(\alpha)}$};
			\draw [-to](A)-- node[above] {$k^{(i,l)}$} (B);
			\draw [-to](B)--node[above]{$\varphi^i$}(C);
			\draw [-to](C)--node[above]{$\phi_l\otimes{\rm id}_A$}(D);
			\draw[-to] (D)--node[above]{$\theta$}(E);
			\draw [-to](0.1,-0.2)  arc (250:290:16.9) ;
			\node (F) at (6,-1.5) {$\psi^{(i,l)}$};
			
		\end{tikzpicture}
        
		\par We claim that the maps $\psi^{(i,l)}$ satisfy the following properties,
		\begin{enumerate}[label=(d\arabic*)]
			\item \label{D-ab d1} $\|(1-\sum_{l=0}^m\sum_{i=0}^s \psi^{(i,l)}(1))a \|\leq(s+1)(m+1)\varepsilon$, for all $a\in F_A$;
            
			\item\label{D-ab d2} $\|\alpha_{\infty,t}\circ \psi^{(i,l)}-\psi^{(i,l)} \|\leq \varepsilon$, for all $i=0,\ldots,s$, $l=0,\ldots,m$ and $t\in M$; 
            
			\item\label{D-ab d3} $\| (\psi^{(i,l)}(d_1)\psi^{(i,l)}(d_2)-\psi^{(i,l)}(d_1d_2)\psi^{(i,l)}(1) )a\|\leq 6\varepsilon$, for all $i=0,\ldots,s$, $l=0,\ldots,m$, $d_1,d_2\in F_D$ and $a\in F_A$;
            
			\item\label{D-ab d4} $\|[a,\psi^{(i,l)}(d)] \|\leq\varepsilon$, for all  $i=0,\ldots,s$, $l=0,\ldots,m$ and $d\in F_D$;
            
			\item\label{D-ab d5} $\|[\psi^{(i,l)}(d_1),\psi^{(i',l')}(d_2)]\|\leq\varepsilon$, for all $i,i'=0,\ldots,s$, $l,l'=0,\ldots,m$ with $(i,l)\neq(i'.l')$ and $d_1,d_2\in F_D$.
		\end{enumerate}
        
		\par As for \ref{D-ab d1}, since for all $a\in A$, we have
		\[\sum_{l=0}^m \theta\circ(\phi_l\otimes{\rm id}_A)(1_{C(G/G_{n_0})}\otimes a)=\theta(\sum_{l=0}^m\phi_l(1)\otimes a)=a. \]
		Then \ref{D-ab d1} follows from \ref{D-ab c1} directly.
        
		\par As for \ref{D-ab d2}, it follows easily from \ref{D-ab c2}.
        
		\par As for \ref{D-ab d3}, since we have $\|ea-a\|<\varepsilon$ for all $a\in F_A$, it suffices to prove the claim for $a=e$ and $4\varepsilon$ instead of $6\varepsilon$. Since $\|\alpha_t(e)-e\|<\varepsilon$ for all $t\in M$, we have $\|\varphi^i(e)-\overline{\mu_i}\otimes e\|<\varepsilon$ for all $i=0,\ldots,s$. Furthermore, for any $x\in A$ and $i=0,\ldots,s$, we have 
		\begin{equation}\label{7.2}
			\|\varphi^i(xe)-\varphi^i(x)(1\otimes e) \|\leq \varepsilon \|x\|
		\end{equation}
		Note that if $y\in C(G/G_{n_0})\otimes A$ and $x\in A$ then 
		\begin{equation}\label{7.3}
			(\phi_l\otimes{\rm id}_A)(y\cdot(1_{C(G/G_{n_0})}\otimes x))=(\phi_l\otimes{\rm id}_A)(y)\cdot(1_{F^{(\alpha)}_\infty(A)}\otimes x)
		\end{equation}
		In particular, for any $d\in D$ we have 
		\begin{equation}\label{7.4}
			\psi^{(i,l)}(d)e=\theta\circ(\phi_l\otimes{\rm id}_A)(\varphi^i\circ k^{(i,l)}(d)\cdot (1\otimes e)).
		\end{equation}
        
		Thus, we have
		\begin{align*}
			&\| (\psi^{(i,l)}(d_1)\psi^{(i,l)}(d_2)-\psi^{(i,l)}(d_1d_2)\psi^{(i,l)}(1) )e\|\\
			\overset{\eqref{7.3}}{=}& \|\theta((\phi_l\otimes{\rm id}_A)(\varphi^i(k^{(i,l)}(d_1)))\cdot  (\phi_l\otimes{\rm id}_A)(\varphi^i(k^{(i,l)}(d_2))\cdot(1\otimes e)) ) \\
			-&\theta((\phi_l\otimes{\rm id}_A)(\varphi^i(k^{(i,l)}(d_1d_2)))\cdot  (\phi_l\otimes{\rm id}_A)(\varphi^i(k^{(i,l)}(1))\cdot(1\otimes e)) ) \|\\
			\leq& \|(\phi_l\otimes{\rm id}_A)(\varphi^i(k^{(i,l)}(d_1)))\cdot  (\phi_l\otimes{\rm id}_A)(\varphi^i(k^{(i,l)}(d_2))\cdot(1\otimes e))\\
			-&(\phi_l\otimes{\rm id}_A)(\varphi^i(k^{(i,l)}(d_1d_2)))\cdot  (\phi_l\otimes{\rm id}_A)(\varphi^i(k^{(i,l)}(1))\cdot(1\otimes e)) \|\\
			\overset{\text{Lemma }\ref{lem7.4}}{\leq}& \|\varphi^i(k^{(i,l)}(d_1))\cdot \varphi^i(k^{(i,l)}(d_2))\cdot(1\otimes e)
			- \varphi^i(k^{(i,l)}(d_1d_2))\cdot \varphi^i(k^{(i,l)}(1))\cdot(1\otimes e) \|\\
			\overset{\eqref{7.2}}{\leq}& \|\varphi^i(k^{(i,l)}(d_1))\cdot \varphi^i(k^{(i,l)}(d_2)e)- \varphi^i(k^{(i,l)}(d_1d_2))\cdot \varphi^i(k^{(i,l)}(1)e) \|+2\varepsilon\\
			\overset{\text{Lemma }\ref{lem7.4}}{\leq}& \|k^{(i,l)}(d_1)k^{(i,l)}(d_2)e-k^{(i,l)}(d_1d_2)k^{(i,l)}(1)e \|+2\varepsilon\\
			\overset{(\text{\ref{D-ab b1}})}{\leq} &\|k^{(i,l)}(d_1)k^{(i,l)}(d_2)e-k^{(i,l)}(d_1d_2)e \|+3\varepsilon\\
			\overset{(\text{\ref{D-ab b4}})}{\leq}& 4\varepsilon.
		\end{align*}
        
		\par As for \ref{D-ab d4},  for all $g\in G$, we have $\|[1\otimes a , \varphi^i(k^{(i,l)}(d))](gG_{n_0})\|\leq \varepsilon$, this follows from the definition of $\varphi^i$ and \ref{D-ab b3}. Hence, 
		\begin{align*}
			\|[a,\psi^{(i,l)}(d)] \|\overset{\eqref{7.3}}{=}&\|\theta((\phi_l\otimes{\rm id}_A)([1\otimes a , \varphi^i(k^{(i,l)}(d))])) \|\\
			\overset{\text{Lemma }\ref{lem7.4}}{\leq} & \|[1\otimes a , \varphi^i(k^{(i,l)}(d))] \| \leq \varepsilon.
		\end{align*}
        
		\par Lastly, for \ref{D-ab d5}, we have
		\begin{align*}
			&\|[\psi^{(i,l)}(d_1),\psi^{(i',l')}(d_2)] \|\\
			\leq & \|[(\phi_l\otimes{\rm id}_A)(\varphi^i(k^{(i,l)}(d_1))) ,(\phi_{l'}\otimes{\rm id}_A)(\varphi^{i'}(k^{(i',l')}(d_2))) ]\|\\
			\overset{\text{Lemma }\ref{lem7.5}}{\leq} & \max_{g_1,g_2\in G} \|[\varphi^i(k^{(i,l)}(d_1))(g_1G_n),\varphi^{i'}(k^{(i',l')}(d_2))(g_2G_n)] \|\\
			\leq & \max_{g_1,g_2\in M'}\|[\alpha_{g_1}((k^{(i,l)}(d_1)),\alpha_{g_2}((k^{(i',l')}(d_2))] \|\\
			\overset{(\text{\ref{D-ab b4}})}{\leq}& \varepsilon.
		\end{align*}
		The claim is proved.
        
		\par Since $D$ is separable, there is an increasing sequence $(F_{D,n})_{n\in\mathbb{N}}$ of finite subsets such that the union is dense in the unit ball of $D$. Since $A$ is separable, there is also an increasing sequence $(F_{A,n})_{n\in\mathbb{N}}$ of finite subsets such that the union is dense in the unit ball of $A$. Since $G$ is second countable, then it is $\sigma$-compact and there is also an increasing sequence $(M_n)_{n\in\mathbb{N}}$ of compact subsets such that the union is $G$. For $(\frac{1}{n}, M_n, F_{D,n}, F_{A,n})$, we successively choose c.p.c.\ maps from $D$ to $A_\infty^{(\alpha)}$ satisfying conditions~\ref{D-ab d1} to \ref{D-ab d5}, so we can obtain c.p.c.\ maps $\hat{\psi}^{i,l}\colon D\to A_\infty^{(\alpha)} $ satisfying
		\begin{itemize}
			\item $a=\sum_{l=0}^m\sum_{i=0}^s \hat{\psi}^{(i,l)}(1)a$;
			\item $\alpha_{\infty,t}\circ \hat{\psi}^{(i,l)}=\hat{\psi}^{(i,l)}$;
			\item $\hat{\psi}^{(i,l)}(d_1)\hat{\psi}^{(i,l)}(d_2)a=\hat{\psi}^{(i,l)}(d_1d_2)\hat{\psi}^{(i,l)}(1)a$;
			\item $[a,\hat{\psi}^{(i,l)}(d)]=0$;
			\item $[\hat{\psi}^{(i,l)}(d_1),\hat{\psi}^{(i',l')}(d_2)]=0$,
		\end{itemize}
		for all $i,i'=0,\ldots,s$ and $l,l'=0,\ldots,m$ with $(i,l)\neq (i',l')$, for all $t\in G$, $a\in A$ and $d_1,d_2\in D$. 
        
		\par For all $i=0,\ldots,s$ and $l,l'=0,\ldots,m$, consider the maps 
		\[\Psi^{(i,l)}\colon D\to F_\infty^{(\alpha)}(A)^{\tilde{\alpha}_\infty}  \]
		given by \[\Psi^{(i,l)}(d)=\hat{\psi}^{(i,l)}(d)+{\rm Ann}(A,A_\infty) \] for all $d\in D$. Then, $\Psi^{(i,l)}$ are c.p.c.\ order zero maps from $D$ to $F_\infty^{(\alpha)}(A)^{\tilde{\alpha}_\infty}$ with pairwise commuting images and satisfying the equation
		\[\sum_{i=0}^s\sum_{l=0}^m\Psi^{(i,l)}(1)=1. \]
		Then applying Lemma~\ref{lemd1} and Lemma~\ref{lemd2}, we have that $A\rtimes_\alpha G$ is $D$-absorbing.	
	\end{proof}
	
	\section{Stability of crossed products}\label{sec8}
    
	In this section, we will show that if the actions of residually compact groups have finite Rokhlin dimension, then the corresponding crossed products are stable. We recall that a C*-algebra $A$ is \emph{stable} if $A\otimes \mathcal{K} \cong A$, where $\mathcal{K}$ denotes the compact operators on a separable infinite-dimensional Hilbert space. In \cite{HR98}, a local characterization of stability for $\sigma$-unital C*-algebras is provided, and we will use a reformulation of it presented in \cite{HSWW17}.
    
	\begin{lemma}[{\cite[Lemma~7.1]{HSWW17}}]\label{lem8.1}
		A $\sigma$-unital C*-algebra $A$ is stable if and only if for any $b\in A_+$ and any $\varepsilon>0$, there exists $y\in A$ such that $\|yy^*-b\|<\varepsilon$ and $\|y^2\|<\varepsilon$.
	\end{lemma}

    Intuitively, $y$ in the above lemma ``moves'' $b$ to an approximately orthogonal position. 

    \begin{remark}\label{rmk:lem8.1-sequence}
        By a standard perturbation argument, one can see that Lemma~\ref{lem8.1} still holds if we allow $y$ to be taken from $A_\infty$. 
    \end{remark}
    
	To establish our theorem, we shall need (the ``only if'' half of) the following more refined version of Lemma~\ref{lem8.1} that involves a subalgebra and multiple $y_i$'s that ``moves'' $b$ to a number of mutually approximately orthogonal positions. To this end, we recall that a subalgebra $B$ in a C*-algebra $A$ is \emph{nondegenerate} if $A = \overline{B \cdot A} := \overline{\mathrm{span}\{b a \colon b \in B, a \in A\}}$, or equivalently, any (or equivalently, some) approximate identity of $B$ is also an approximate identity of $A$. 
    
	\begin{lemma}\label{lem8.2}
        Let $A$ be a C*-algebra and let $B$ be a $\sigma$-unital nondegenerate C*-subalgebra of $A$. Let $d$ be a natural number. Then $A$ is stable if and only if for any $b\in B_+$ and $\varepsilon>0$, there exist $e,e_0, \ldots, e_d, c_0, \ldots, c_d\in B_+^1$ and $y_0 , \ldots, y_d\in A$ such that for any $i,j \in \{0,\ldots,d\}$, we have $\|eb-b\|<\varepsilon$, 
        $e_i e = e$, 
        $\|y_iy_i^*-e_i\|<\varepsilon$, 
        $\|y_i^*y_ic_i-y_i^*y_i\|<\varepsilon$, 
        $\|y_i^2\|<\varepsilon$, 
        $e c_i=0$, 
        $(e_i + c_i) e_i = e_i$, 
        and
        $c_i c_j=0$ whenever $j\neq i$.
	\end{lemma}
    
	\begin{proof}
		Note that $A$ is also $\sigma$-unital since it has a $\sigma$-unital nondegenerate C*-subalgebra. Then, to prove the ``if'' direction, one observes that Lemma~\ref{lem8.1} holds true if one requires in addition that $b$ is taken from some fixed approximate unit of $A$. Applying this observation to a fixed approximate unit $(f_\lambda)_{\lambda \in \Lambda}$ of $B$, one quickly sees that our assumption verifies the condition in Lemma~\ref{lem8.1} (with a different $\varepsilon$) by taking $y := b^{\frac{1}{2}} y_0$. 

        Now we prove the ``only if '' direction, which is a bit more involved. 
        Since $B$ is $\sigma$-unital, it admits a countable approximate identity $\{f_n\}_{n\in\mathbb{N}}$ that satisfies $f_n f_{n+1}= f_n= f_{n+1} f_n$ for any $n \in\mathbb{N}$. This also forms an approximate identity of $A$ since $B$ is nondegenerate in $A$. 
        Now we fix arbitrary $b\in B_+$ and $\varepsilon>0$, and assume, without loss of generality, $\varepsilon<1$. We shall find a strictly increasing chain $n_0<n_1 <\ldots < n_d < n_{d+1}$ of natural numbers and elements $y_0, \ldots, y_d \in A$ by a recursive procedure. To begin with, we find $n_{0} \in \mathbb{N}$ such that $\|f_{n_{0}} b-b\|<\varepsilon$. Then for each $i \in \{0, \ldots, d\}$, supposing we have determined $n_{i} \in \mathbb{N}$, we go on to determine $y_i \in A$ and  $n_{i+1} \in \mathbb{N}$: writing $n_i^+$ for $n_i + 1$ for convenience and applying Lemma~\ref{lem8.1}, we find $y_i\in A$ such that $\|y_i y_i^*- f_{n_{i} ^+}\|<\varepsilon/5$ and $\|y_i^2\|<\varepsilon/5$; we then can find $n_{i+1} > n_i^+$ such that $\|y_i^*y_i f_{n_{i+1}}-y_i^*y_i\|<\varepsilon/5$. 
        
        Having completed the recursive definitions, we set 
        \[
            e := f_{n_0}, \quad e_i := f_{n_{i}^+} \quad \text{and} \quad c_i := f_{n_{i+1}} - f_{n_{i}^+} \quad \text{ for } i \in \{0, \ldots, d\} \; .
        \]
    It remains to verify they satisfy the desired conditions. It is clear from our construction that  $\|eb-b\|<\varepsilon$, 
        $e_i e = e$, 
        $\|y_iy_i^*-e_i\|<\varepsilon$, 
        $\|y_i^2\|<\varepsilon$, 
        for any $i \in \{0, \ldots, d\}$. 
        where $i,j=0,\ldots,d$ and $i\neq j$.
        We also check, for any $i \in \{0, \ldots, d\}$, that 	
        \[ec_i=f_{n_0} (f_{n_{i+1}} - f_{n_{i}^+})=f_{n_0}-f_{n_0}=0 \; , \]	
        \[
        (e_i + c_i) e_i = f_{n_{i+1}}  f_{n_{i}^+} = f_{n_{i}^+} = e_i \; ,
        \]
        \[c_ic_j=(f_{n_{i+1}} - f_{n_{i}^+})(f_{n_{j+1}} - f_{n_{j}^+})=f_{n_{i+1}} - f_{n_{i}^+}-f_{n_{i+1}} + f_{n_{i}^+}=0  \quad \text{ for any } j > i  \; , \]
	and  
		\begin{align*}
		&\|y_i^*y_ic_i-y_i^*y_i\|=\|y_i^*y_i(f_{n_{i+1}} - f_{n_{i}^+})-y_i^*y_i\|\\
		&=\|(y_i^*y_i f_{n_{i+1}}-y_i^*y_i) - y_i^*y_i(f_{n_{i}^+} - y_iy_i^* )-y_i^*y_iy_iy_i^* \|\\
		&<\varepsilon/5+\|y_i\|^2\varepsilon/5+\|y_i\|^2\varepsilon/5\\
		&< \varepsilon/5+2\varepsilon/5+2\varepsilon/5=\varepsilon \; .
	\end{align*}
    This completes the proof.
	\end{proof}

	\par Now, we proceed to prove the main theorem in this section.

	\begin{theorem}\label{thm: stable}
		Let $G$ be a second countable residually compact group with a non-open cocompact closed  subgroup $H$ (in particular, $G$ is not discrete). Let $A$ be a separable C*-algebra and let $\alpha\colon  G\to {\rm Aut}(A)$ a continuous group action. If ${\rm dim}_{\rm Rok}(\alpha,H)<\infty$, then $A\rtimes_\alpha G$ is stable. 
	\end{theorem}
    
    \begin{proof}
    	Because $A$ is separable and $G$ is second countable, $A\rtimes_\alpha G$ is $\sigma$-unital. We will use Lemma~\ref{lem8.1} and Remark~\ref{rmk:lem8.1-sequence} to prove that it is stable. To this end, we fix  $b\in (A\rtimes_\alpha G)_+$ and $\varepsilon>0$. Without loss of generality, we may assume $\|b\|\leq 1$ and $\varepsilon<1$. 
   There exist $a\in A_+^1$ (viewed as an element in $\mathcal{M}(A \rtimes_{\alpha} G)$) and $g\in C^*(G)_+^1$ such that 
   \begin{equation}\label{Stable: bag=b}
       \|b^{\frac{1}{2}}a\theta_{\alpha}(g)-b^{\frac{1}{2}}\|<\varepsilon
   \end{equation}
   where $\theta_{\alpha}$ denotes the canonical homomorphism from $C^*(G)$ to $\mathcal{M}(A\rtimes_\alpha G)$. 
        
    	\par       
        As $G/H$ is compact, the unital embedding of $\mathbb{C}$ in $C(G/H)$ induces a canonical homomorphism $C^*(G)\to C(G/H)\rtimes_\lambda G$, which is in fact injective, because for any unitary representation $\rho$ of $G$, the embedding of the trivial representation of $G$ as the constants in the quasi-regular representation $\lambda_{G/H} \colon G \to U(L^2(G/H))$ gives rise to a containment $\rho \leq \rho \otimes \lambda_{G/H}$ of unitary representations of $G$, while $\rho \otimes \lambda_{G/H}$ is covariant with the standard representation of $C(G/H)$ on $L^2(G/H)$ and thus extends to a representation of $C(G/H)\rtimes_\lambda G$. Hence we may view $C^*(G)$ as a subalgebra of $C(G/H)\rtimes_\lambda G$ that clearly is nondegenerate. It is also $\sigma$-unital as $G$ is second countable. 
        
        By Green's imprimitivity, $C(G/H)\rtimes_\lambda G\cong C^*(H)\otimes K(L^2(G/H))$, which shows it is stable, since $H$ is not open. Therefore, we can regard $g$ as a positive element in $C^*(G)\subseteq C(G/H)\rtimes_\lambda G$ and use Lemma~\ref{lem8.2} to obtain 
        $e,e_0, \ldots, e_d, c_0, \ldots, c_d\in C^*(G)_+^1$ and $y_0 , \ldots, y_d\in C(G/H)\rtimes_\lambda G$ such that
        for any $i,j \in \{0,\ldots,d\}$, we have 
        \begin{gather}
            \|eg-g\|<\varepsilon \label{Stable: eg=g}\\
            e_i e = e\label{Stable: ee_i=e}\\
            \|y_iy_i^*-e_i\|<\varepsilon\label{Stable: y_iy*_i=e_i}\\
            \|y_i^*y_ic_i-y_i^*y_i\|<\varepsilon\label{Stable: y*_iy_ic_i=y*_iy_i}\\
            \|y_i^2\|<\varepsilon\label{Stable: y_i^2=0}\\
            e c_i=0\label{Stable: ec_i=0}\\
            (e_i + c_i) e_i = e_i\label{Stable: (e_i + c_i) e_i = e_i}\\
            c_i c_j=0\;\text{whenever}\;j\neq i \label{Stable: c_i c_j=0}.
        \end{gather}
        It follows that
        \begin{gather}
            \|y_i\|<\sqrt{1+\varepsilon}<2 \label{Stable: |y_i|<2}\\
            \|y_ic_i-y_i\| < 2\sqrt{\varepsilon}\label{Stable: y_ic_i=y_i}\\
            \|(e_i+c_i)y_i-y_i\|<3\sqrt{\varepsilon}\label{Stable: (e_i+c_i)y_i=y_i}
        \end{gather}
             
    	\par We write $d := {\rm dim}_{\rm Rok}(\alpha,H) < \infty$ and denote the $G$-shift on $G/H$ by $\lambda$. Then there exist $\lambda$-$\tilde{\alpha}_\infty$ equivariant c.p.c.\ order zero maps $\varphi_0,\ldots,\varphi_d\colon C(G/H)\to F^{(\alpha)}_\infty(A)$ with  $\varphi_0(1)+\cdots+\varphi_d(1)=1$. By \cite[Corollary~4.2]{WZ09}, for each $i=0,\ldots,d$, we define $\psi_i=\varphi_i^{\frac{1}{2}}$, and then we obtain $\lambda$-$\tilde{\alpha}_\infty$ equivariant c.p.c.\ order zero maps $$\psi_0,\ldots,\psi_d\colon C(G/H)\to F^{(\alpha)}_\infty(A)$$ with  $\psi_0(1)^2+\cdots+\psi_d(1)^2=1$. Hence, by Lemma~\ref{lem2.31} there are induced c.p.c.\ order zero maps $$\psi_0\rtimes G,\ldots,\psi_d\rtimes G\colon C(G/H)\rtimes_\lambda G\to F^{(\alpha)}_\infty(A)\rtimes_{\tilde{\alpha}_\infty}G \; ,$$
        which are also $C^*(G)$-bimodule maps in the sense that for any $i \in \{0,\ldots,d\}$, any $x \in C^*(G)$ and any $y \in C(G/H)\rtimes_\lambda G$, we have 
        \[
            (\psi_i\rtimes G)(x y) = \theta_{\tilde{\alpha}_\infty}(x) \ (\psi_i\rtimes G)(y) \quad \text{and} \quad (\psi_i\rtimes G)(y x) =  (\psi_i\rtimes G)(y) \ \theta_{\tilde{\alpha}_\infty}(x) \; ,
        \]
        where $\theta_{\tilde{\alpha}_\infty}$ denotes the canonical homomorphism from $C^*(G)$ to $\mathcal{M}(F^{(\alpha)}_\infty(A)\rtimes_{\tilde{\alpha}_\infty}G)$. 
        It follows that for any different $i, j \in \{0, \ldots, d\}$, we have
        \begin{equation}\label{Stable: psi_i(y_i)psi_jy_j=0}
            \begin{aligned}
                &\|(\psi_i\rtimes G)(y_i)\ (\psi_j\rtimes G)(y_j^*)\|\\
    		&\overset{\eqref{Stable: y_ic_i=y_i}}{<} 8\sqrt{\varepsilon}+\|(\psi_i\rtimes G)(y_ic_i)\ (\psi_j\rtimes G)(c_jy_j^*)\| \\
    		&=8\sqrt{\varepsilon}+\|(\psi_i\rtimes G)(y_i) \ \theta_{\tilde{\alpha}_\infty}(c_ic_j)\ (\psi_j\rtimes G)(y_j^*)\|\\
    		&=8\sqrt{\varepsilon}.
            \end{aligned}
        \end{equation}

    Then, \begin{align*}
    	&\left\|\theta_{\tilde{\alpha}_\infty}(g)\left(\sum_{i=0}^d(\psi_i\rtimes G)(y_i)\right)\left(\sum_{i=0}^d(\psi_i\rtimes G)(y_i)\right)^*\theta_{\tilde{\alpha}_\infty}(g)-\theta_{\tilde{\alpha}_\infty}(g)^2\right\|\\
    	&\overset{\eqref{Stable: psi_i(y_i)psi_jy_j=0}}{<}8d(d+1)\sqrt{\varepsilon}+	\left\|\theta_{\tilde{\alpha}_\infty}(g)\left(\sum_{i=0}^d(\psi_i\rtimes G)(y_i)(\psi_i\rtimes G)(y_i^*)\right)\theta_{\tilde{\alpha}_\infty}(g)-\theta_{\tilde{\alpha}_\infty}(g)^2\right\|\\
    	&\overset{\eqref{Stable: (e_i+c_i)y_i=y_i}}{<}(8d+6)(d+1)\sqrt{\varepsilon}+ \left\|\theta_{\tilde{\alpha}_\infty}(g)\left(\sum_{i=0}^d(\psi_i\rtimes G)(y_iy_i^*)(\psi_i\rtimes G)(e_i+c_i)\right)\theta_{\tilde{\alpha}_\infty}(g)-\theta_{\tilde{\alpha}_\infty}(g)^2\right\|\\
    	&\overset{\eqref{Stable: y_iy*_i=e_i}}{<}(8d+6)(d+1)\sqrt{\varepsilon}+2(d+1)\varepsilon+ \left\|\theta_{\tilde{\alpha}_\infty}(g)\left(\sum_{i=0}^d(\psi_i\rtimes G)(e_i)(\psi_i\rtimes G)(e_i+c_i)\right)\theta_{\tilde{\alpha}_\infty}(g)-\theta_{\tilde{\alpha}_\infty}(g)^2\right\|\\
    	&=(8d+6)(d+1)\sqrt{\varepsilon}+2(d+1)\varepsilon+\left\|\sum_{i=0}^d\psi_i(1)^2\theta_{\tilde{\alpha}_\infty}(ge_i(e_i+c_i)g-g^2)\right\|\\
        &\overset{\eqref{Stable: (e_i + c_i) e_i = e_i}}{=} (8d+6)(d+1)\sqrt{\varepsilon}+2(d+1)\varepsilon+\left\|\sum_{i=0}^d\psi_i(1)^2\theta_{\tilde{\alpha}_\infty}(ge_ig-g^2)\right\|\\
        &\overset{\eqref{Stable: eg=g}}{<} (8d+6)(d+1)\sqrt{\varepsilon}+4(d+1)\varepsilon+\left\|\sum_{i=0}^d\psi_i(1)^2\theta_{\tilde{\alpha}_\infty}(gee_ieg-g^2)\right\|\\
    	&\overset{\eqref{Stable: ee_i=e}}{<}(8d+6)(d+1)\sqrt{\varepsilon}+4(d+1)\varepsilon+\|geeg-g^2\|\\
    	&\overset{\eqref{Stable: eg=g}}{<} (8d+6)(d+1)\sqrt{\varepsilon}+(4d+6)\varepsilon.
        \end{align*}
    And \begin{align*}
	&\left\|\theta_{\tilde{\alpha}_\infty}(g)\left(\sum_{i=0}^d(\psi_i\rtimes G)(y_i)\right)\theta_{\tilde{\alpha}_\infty}(g) \right\|\\
	&\overset{\eqref{Stable: y_ic_i=y_i}}{<}2(d+1)\sqrt{\varepsilon}+\left\|\theta_{\tilde{\alpha}_\infty}(g)\left(\sum_{i=0}^d(\psi_i\rtimes G)(y_ic_i)\right)\theta_{\tilde{\alpha}_\infty}(g)\right\|\\
	&\overset{\eqref{Stable: eg=g}}{<}2(d+1)\sqrt{\varepsilon}+2(d+1)\varepsilon+\left\|\theta_{\tilde{\alpha}_\infty}(g)\left(\sum_{i=0}^d(\psi_i\rtimes G)(y_i) \ \theta_{\tilde{\alpha}_\infty}(c_i)\right)\theta_{\tilde{\alpha}_\infty}(eg)\right\|\\
	&\overset{\eqref{Stable: ec_i=0}}{=}2(d+1)\sqrt{\varepsilon}+2(d+1)\varepsilon.
        \end{align*}

   \par  
   We lift each $(\psi_i\rtimes G)(y_i)$ to $x_i\in (A^{(\alpha)}_\infty\cap A') \rtimes_{\alpha_\infty} G$
   and denote the canonical homomorphism from $C^*(G)$ to $\mathcal{M}(A_\infty^{(\alpha)} \rtimes_{\alpha_\infty}G)$ by $\theta_{\alpha_\infty}$.  
   Viewing $a\in A$ as an element of $\mathcal{M}(A^{(\alpha)}_\infty\rtimes_{\alpha_\infty} G)$, we have 
   \begin{equation}\label{Stable: agxxga=agga}
       \left\|a\theta_{\alpha_\infty}(g)\left(\sum_{i=0}^dx_i\right)\left(\sum_{i=0}^dx_i^*\right)\theta_{\alpha_\infty}(g)a-a\theta_{\alpha_\infty}(g)^2a \right\|<(8d+6)(d+1)\sqrt{\varepsilon}+(4d+6)\varepsilon,
   \end{equation}
   and
   \begin{equation}\label{Stable: agxga=0}
       \left\|a\theta_{\alpha_\infty}(g)\left(\sum_{i=0}^dx_i\right)\theta_{\alpha_\infty}(g)a\right\|<2(d+1)\sqrt{\varepsilon}+2(d+1)\varepsilon.
   \end{equation}

    	\par By Lemma~\ref{lem2.21}, there is a natural $\ast$-homomorphism  $\phi\colon  A_\infty^{(\alpha)} \rtimes_{\alpha_\infty}G \to (A\rtimes_\alpha G)_\infty$. Viewing $b,a\theta_{\alpha}(g)\in A\rtimes_\alpha G$ as elements in $(A\rtimes_\alpha G)_\infty$ via the canonical embedding, we set
    	\[y=b^{\frac{1}{2}}a\theta_{\alpha}(g)\sum_{i=0}^d\phi(x_i)\in (A\rtimes_\alpha G)_\infty .\]
    	Then
    	\begin{align*}
    		\|yy^*-b\|&=\left\|b^{\frac{1}{2}}a\theta_{\alpha}(g)\left(\sum_{i=0}^d\phi(x_i)\right)\left(\sum_{i=0}^d\phi(x_i^*)\right)\theta_{\alpha}(g)ab^{\frac{1}{2}}-b\right\|\\
    		&\overset{\eqref{Stable: agxxga=agga}}{<}(8d+6)(d+1)\sqrt{\varepsilon}+(4d+6)\varepsilon+\|b^{\frac{1}{2}}a\theta_{\alpha}(g)\theta_{\alpha}(g)ab^{\frac{1}{2}}-b\|\\
    		&\overset{\eqref{Stable: bag=b}}{<}(8d+6)(d+1)\sqrt{\varepsilon}+(4d+8)\varepsilon\\
            &<(8d^2+18d+14)\sqrt{\varepsilon}.
    	\end{align*}
        In particular, $\|y^2\|\leq 1+(8d^2+18d+14)\sqrt{\varepsilon}$. 
    We further have
    \begin{align*}
    	\|y^2\|&= \left\|b^{\frac{1}{2}}a\theta_{\alpha}(g)\left(\sum_{i=0}^d\phi(x_i)\right)b^{\frac{1}{2}}a\theta_{\alpha}(g)\left(\sum_{i=0}^d\phi(x_i)\right) \right\|\\
    	&\overset{\eqref{Stable: bag=b}}{<} (d+1)\|y\|\varepsilon+ \left\|b^{\frac{1}{2}}a\theta_{\alpha}(g)\left(\sum_{i=0}^d\phi(x_i)\right)\theta_{\alpha}(g)ab^{\frac{1}{2}}a\theta_{\alpha}(g)\left(\sum_{i=0}^d\phi(x_i)\right) \right\|\\
        &\overset{\eqref{Stable: agxga=0}}{<} (d+1)\|y\|\varepsilon+ (2(d+1)\sqrt{\varepsilon}+2(d+1)\varepsilon)\|y\|\\
    	&<(5(d+1)\sqrt{\varepsilon}) \sqrt{1+(8d^2+18d+14)\sqrt{\varepsilon}} .
    \end{align*}

  Therefore, by Lemma~\ref{lem8.1} and Remark~\ref{rmk:lem8.1-sequence}, $A\rtimes_\alpha G$ is stable. 
    \end{proof}

	From the above theorem, we also have an immediate corollary concerning residually compact groups.
    
	\begin{corollary}
			Let $G$ be a second countable residually compact group with a non-open cocompact closed  subgroup (in particular, $G$ is not discrete). Let $A$ be a separable C*-algebra and $\alpha\colon  G\to {\rm Aut}(A)$ a continuous group action. If ${\rm dim}_{\rm Rok}(\alpha)<\infty$, then $A\rtimes_\alpha G$ is stable. 
	\end{corollary}

\section{Tube dimension vs. Rokhlin dimension}\label{sec9}

In \cite{HSWW17}, the Rokhlin dimension and the tube dimension of a flow have a mutually controlling relationship. Enstad, Favre and Raum extended the definition of tube dimension to locally compact groups in \cite{EFR23}. So in this section, we study the relationship with Rokhlin dimension and the tube dimension for actions of residually compact groups. We begin by recalling the definition of the tube dimension from \cite{EFR23}.

\begin{definition}[{\cite[Definition~3.1]{EFR23}}]
	Let $G$ be a locally compact group, and $X$ be a locally compact Hausdorff space. Let $G\curvearrowright X$ be an action and $K\subseteq G$ be some compact identity neighborhood. A \emph{$K$-slice} is a compact subset $E\subseteq X$ such that the map $K\times E\to X$ given by $(g,x)\mapsto gx$ is injective. The resulting image $KE$ in $X$ is called a \emph{tube}.
\end{definition}

From the definition, it is easy to see that the tube $KE$ is a compact set in $X$ homeomorphic to $K\times E$. 

\begin{definition}[{\cite[Definition~4.1]{EFR23}}]
	Let $G$ be a locally compact group, and let $X$ be a locally compact Hausdorff space. The \emph{tube dimension} of an action $G\curvearrowright X$, denoted by ${\rm dim}_{\rm tube}(G\curvearrowright X)$, is the least natural number $d$ such that for all compact subsets $K\subseteq G$ and $Y\subseteq X$ there is a family $\mathcal{U}$ of open sets of $X$ satisfying the following properties:
	\begin{enumerate}
		\item for all $x\in Y$ there is $U\in \mathcal{U}$ such that $Kx\subseteq U$;
		\item every $U\in \mathcal{U}$ is contained in a tube;
		\item the multiplicity of $\mathcal{U}$ is at most $d+1$.
	\end{enumerate}
If no such natural number $d$ exists, then we define ${\rm dim}_{\rm tube}(G\curvearrowright X)=\infty$.
\end{definition}

For convenience, we also recall some propositions related to the tube dimension, which are given in \cite{EFR23}. We say that a function $\varphi\colon X\to \mathbb{C}$ is $(K,\varepsilon)$-F{\o}lner if $|\varphi(gx)-\varphi(x)|<\varepsilon$, for every $g\in K$, $x\in X$, where $K$ is a compact subset of $G$. 

\begin{proposition}[{\cite[Proposition~5.4]{EFR23}}]\label{prop9.3}
	Let $G$ be an amenable locally compact second countable group, $X$ be a locally compact Hausdorff space, and $G\curvearrowright X$ be an action. Then the following are equivalent.
	\begin{enumerate}
		\item The tube dimension of $G\curvearrowright X$ is at most $d$.
		\item For every compact subset $K\subseteq G$, $\varepsilon>0$ and every compact subset $Y\subseteq X$ there is a finite partition of unity $\{\varphi_i\}_{i\in I}$ for $Y\subseteq X$ such that
		\begin{enumerate}[label=(2\alph*)]
			\item\label{tube: folner} $\varphi_i$ is $(K,\varepsilon)$-F{\o}lner for all $i\in I$;
			\item\label{tube: interior of a tube} $\varphi_i$ is supported in the interior of a tube;
			\item\label{tube: partition} there is a partition $I=I^0\sqcup\cdots\sqcup I^d$ such that for all $l\in \{0,\ldots,d\}$ and all $i,j\in I^l$ we have \[{\rm supp}(\varphi_i)\cap {\rm supp}(\varphi_j)=\emptyset. \]
		\end{enumerate}
	\end{enumerate}
	
\end{proposition}

Using the proposition shown above, we can also demonstrate a close relationship between the tube dimension of a topological dynamical system and the Rokhlin dimension of the induced C*-dynamical system in the residually compact case.

\begin{theorem}\label{thm: tubedim}
	Let $G$ be an amenable residually compact group and let $X$ be a locally compact Hausdorff space. Let $G\curvearrowright X$ be an action and $\alpha\colon  G\to {\rm Aut}(C_0(X))$ be the associated action. Then
	\[{\rm dim}_{\rm Rok}(\alpha)\leq {\rm dim}_{\rm tube}(G\curvearrowright X). \]
\end{theorem}

\begin{proof}
	If the tube dimension of $G\curvearrowright X$ is infinity, there is nothing to prove. So, let us assume ${\rm dim}_{\rm tube}(G\curvearrowright X)=d$ for some positive integer $d$, and we show that ${\rm dim}_{\rm Rok}(\alpha)\leq d$. From the definition of the Rokhlin dimension, it is sufficient to show that ${\rm dim}_{\rm Rok}(\alpha,H)\leq d$ for each closed cocompact subgroup $H$ of $G$. Fixing a $H$, we denote the $G$-shift on $G/H$ by $\lambda$ and use Proposition~\ref{prop5.1} to prove ${\rm dim}_{\rm Rok}(\alpha,H)\leq d$. 
    
	\par For any $\varepsilon>0$ and finite sets $M\subseteq G$, $F\subseteq C_0(X)$ and $S\subseteq C(G/H)$, without loss of generality, we assume $F\subseteq C_c(X)$ and the norm of the elements in $F$ and $S$ are less than 1. Then we can pick a compact set $Y\subseteq X$ so large that it contains the supports of all functions in $F$. Then we apply Proposition~\ref{prop9.3} to obtain a finite partition of unity $\{\varphi_i\}_{i\in I}$ of $Y$ and a decomposition $I=I^0\sqcup\cdots\sqcup I^d$ satisfying the three conditions with regard to $\varepsilon/2$ and $M^{-1}$. From condition~\ref{tube: interior of a tube}, each $\varphi_i$ is supported in the interior of a tube, we denote it by $B_i$. Thus, there exist compact identity neighborhood $K_i$ and compact subset $E_i$ such that $B_i=K_iE_i\cong K_i \times E_i$. Hence, there is a continuous map $\pi_i$ from $B_i$ to $G/H$ given by
	\[B_i\cong K_i\times E_i \to K_i \to G \to G/H. \]
	Moreover, there is an induced unital homomorphism
	\[\overline{\pi_i}\colon C(G/H)\to C(B_i), \overline{\pi_i}(f)=f\circ \pi_i, f\in C(G/H). \]
	Now set 
	\[\psi_i=\varphi_i\overline{\pi_i} \colon C(G/H)\to C(B_i), \]
	$\psi_i$ can be continuously extended to all $X$ by setting $\psi_(x)=0$ for all $x\in X\backslash B_i$. Then we define
	\[\phi_l=\sum_{i\in I^l}\psi_i \colon C(G/H)\to C_0(X) \]
	for all $l=0,\ldots,d$. Note that for any $x\in X$ and $f\in C(G/H)$, at most one of the functions $\psi_i(f)$ in the sum is nonzero because they are pairwise orthogonal from condition~\ref{tube: partition}. It is also obvious that each $\phi_l$ is completely positive contractive. Subsequently, we check that $\{\phi_l\colon l=0,\ldots,d\}$ satisfies the conditions~\ref{prop:A 1}, \ref{prop:A order 0}, \ref{prop:A equiva} and \ref{prop:A abel} in Proposition~\ref{prop5.1}.
    
	\par As for condition~\ref{prop:A 1}, by definition, we have
	\begin{align*}
		\sum_{l=0}^d \phi_l(1)=\sum_{l=0}^d\sum_{i\in I^l}\varphi_i\overline{\pi_i}(1)
		=\sum_{i\in I}\varphi_i.
	\end{align*}
  For any $a\in F$, $a$ is supported in $Y$ and $\{\varphi_i\}$ is the partition of unity for $Y$. Therefore, \[\sum_{l=0}^d \phi_l(1)a=\sum_{i\in I}\varphi_ia=a. \]
  
  \par As for condition~\ref{prop:A order 0}, fix $l\in \{0,\ldots,d\}$ and let $f_1,f_2,f_3\in C(G/H)$, we have
  \begin{align*}
  	\phi_l(f_1)\phi_l(f_2f_3)&=\left(\sum_{i\in I^l}\varphi_i\overline{\pi_i}(f_1)\right)\left(\sum_{i\in I^l}\varphi_i\overline{\pi_i}(f_2f_3)\right)\\
  	&=\sum_{i\in I^l}\varphi_i^2\overline{\pi_i}(f_1)\overline{\pi_i}(f_2f_3)\\
  	&=\sum_{i\in I^l}\varphi_i^2\overline{\pi_i}(f_1f_2)\overline{\pi_i}(f_3)\\
  	&=\left(\sum_{i\in I^l}\varphi_i\overline{\pi_i}(f_1f_2)\right)\left(\sum_{i\in I^l}\varphi_i\overline{\pi_i}(f_3)\right)\\
  	&=\phi_l(f_1f_2)\phi_l(f_3).
  \end{align*}
Thus for any $a\in F$ and $f_1,f_2,f_3\in S$, we have
$	\phi_l(f_1)\phi_l(f_2f_3)a=\phi_l(f_1f_2)\phi_l(f_3)a$.

\par As for condition~\ref{prop:A equiva}, fix a $l\in \{0,\ldots,d\}$ and let $f\in C(G/H), g\in M$ and $x\in X$.
\begin{align*}
	&(\alpha_g(\phi_l(f))-\phi_l(\lambda_g(f)))(x)\\
	&=\sum_{i\in I^l}\alpha_g(\varphi_i\overline{\pi_i}(f))(x)-\sum_{i\in I^l}\varphi_i\overline{\pi_i}(\lambda_g(f))(x)\\
	&=\sum_{i\in I^l}\varphi_i(g^{-1}x)f(\pi_i(g^{-1}x))-\sum_{i\in I^l}\varphi_i(x)f(g^{-1}\pi_i(x))\\
\end{align*}
Since at most one of the $\varphi_i$ in the sum is nonzero, there exist $i_0,j_0\in I^l$ such that 
\[(\alpha_g(\phi_l(f))-\phi_l(\lambda_g(f)))(x)=\varphi_{i_0}(g^{-1}x)f(\pi_{i_0}(g^{-1}x))-\varphi_{j_0}(x)f(g^{-1}\pi_{j_0}(x)). \]
This breaks down into two case:
\begin{enumerate}
	\item[Case 1:] If $i_0=j_0$, then $g^{-1}x$ and $x$ are both in $B_{i_0}$. Thus, by the definition of $\pi_i$,  $\pi_{i_0}(g^{-1}x)=g^{-1}\pi_{j_0}(x)$. Therefore, 
	\begin{align*}
		&|\varphi_{i_0}(g^{-1}x)f(\pi_{i_0}(g^{-1}x))-\varphi_{j_0}(x)f(g^{-1}\pi_{j_0}(x))|\\
		&=|(\varphi_{i_0}(g^{-1}x)-\varphi_{i_0}(x))f(\pi_{i_0}(g^{-1}x))|\\
		&\leq |\varphi_{i_0}(g^{-1}x)-\varphi_{i_0}(x)|<\varepsilon/2.
	\end{align*}
    \item[Case 2:] If $i_0\neq j_0$, then $\varphi_{i_0}(x)=0$ and $\varphi_{j_0}(g^{-1}x)=0$. Hence, 
    \begin{align*}
    	&|\varphi_{i_0}(g^{-1}x)f(\pi_{i_0}(g^{-1}x))-\varphi_{j_0}(x)f(g^{-1}\pi_{j_0}(x))|\\
    	&\leq |(\varphi_{i_0}(g^{-1}x)-\varphi_{i_0}(x))f(\pi_{i_0}(g^{-1}x))|+|(\varphi_{j_0}(x)-\varphi_{j_0}(g^{-1}x))f(g^{-1}\pi_{j_0}(x))|\\
    	&\leq |\varphi_{i_0}(g^{-1}x)-\varphi_{i_0}(x)|+|\varphi_{j_0}(x)-\varphi_{j_0}(g^{-1}x)|<\varepsilon.
    \end{align*}
\end{enumerate}
Thus, we have 
\[\|\alpha_g(\phi_l(f))-\phi_l(\lambda_g(f))\|<\varepsilon. \]

\par Condition~\ref{prop:A abel} is trivially satisfied, since $C_0(X)$ is commutative. So, ${\rm dim}_{\rm Rok}(\alpha,H)\leq d$.
\end{proof}

The tube dimension provides a better estimate of the nuclear dimension of the associated crossed product than the Rokhlin dimension. However, there actually exist non-commutative C*-algebras that admit residually compact group actions with finite Rokhlin dimension, to which the tube dimension cannot be applied for estimation.

\begin{example}
	Let $X$ be a compact Hausdorff space, $G$ a residually compact group and $\Gamma$ a discrete group. Let $\alpha$ be an action of $G$ on $X$, $\beta$ an action of $\Gamma$ on $X$ and the actions $\alpha,\beta$ are commutative. Denote the induced actions on $C(X)$ by $\bar{\alpha}$ and $\bar{\beta}$ respectively. Then there is a natural action of $G$ on the crossed product $C(X)\rtimes_{\bar{\beta}}\Gamma$, we denote it by $\gamma$. Since $C(X)\rtimes_{\bar{\beta}}\Gamma$ is unital, we can define the infinite tensor product $(C(X)\rtimes_{\bar{\beta}}\Gamma)^{\otimes \infty}$ and there is also a natural action $\gamma^{\otimes \infty}$ on the infinite tensor product. 
    
	\par If $\bar{\alpha}$ has a finite Rokhlin dimension, we claim that $\gamma^{\otimes \infty}$ has finite Rokhlin dimension. We assume ${\rm dim}_{\rm Rok}(\bar{\alpha})=n$, and then for any closed cocompact subgroup $H$ of $G$, there exist $n+1$  $G\text{-}{\rm shift}$-$\bar{\alpha}_\infty$ equivariant c.p.c.\ order zero maps $$\phi_0,\ldots,\phi_n\colon C(G/H)\to C(X)_\infty$$ with $\phi_0(1)+\cdots+\phi_n(1)=1$. Since $C(X)$ can be embedded in $C(X)\rtimes_{\bar{\beta}}\Gamma$ which can be embedded in $(C(X)\rtimes_{\bar{\beta}}\Gamma)^{\otimes \infty}$, there is a unital $\bar{\alpha}_\infty$-$(\gamma^{\otimes \infty})_\infty$ equivariant homomorphism $\rho$ from $C(X)_\infty$ to $((C(X)\rtimes_{\bar{\beta}}\Gamma)^{\otimes \infty})_\infty$. Actually, $\rho$ maps  $C(X)_\infty$ into $((C(X)\rtimes_{\bar{\beta}}\Gamma)^{\otimes \infty})_\infty\cap (C(X)\rtimes_{\bar{\beta}}\Gamma)^{\otimes \infty})' $. 
    
	Then we define $$\psi_i=\rho \circ \phi_i,$$ $i=0,\ldots,n$, $\{\psi_i\}$ are $G\text{-shift}$-$(\gamma^{\otimes \infty})_\infty$ equivariant c.p.c.\ order zero maps from $C(G/H)$ to $((C(X)\rtimes_{\bar{\beta}}\Gamma)^{\otimes \infty})_\infty\cap (C(X)\rtimes_{\bar{\beta}}\Gamma)^{\otimes \infty})' $ with $\psi_0(1)+\cdots+\psi_n(1)=1$. Thus, $\gamma^{\otimes \infty}$ has finite Rokhlin dimension.
\end{example}
	
\bibliographystyle{alpha}
\bibliography{ref}

\end{document}